\documentclass[12pt,a4paper]{amsart}

\usepackage{amsfonts}
\usepackage{amssymb}
\usepackage{amsthm}
\usepackage{amsmath}
\usepackage{mathrsfs}
\usepackage{bbold}
\usepackage{enumerate}
\usepackage{hyperref}

\usepackage{graphicx}

\usepackage[all,cmtip]{xy}

\def\IC{\mathbb{C}}

\def\IN{\mathbb{N}}

\def\BH{\mathcal{B}(H)}

\newcommand{\id}{\mathop{\mathrm{id}}}
\newcommand{\rank}{\mathop{\mathrm{rank}}}
\newcommand{\TRO}{\mathrm{TRO}}

\newcommand{\TROu}{{T^*}}

\newcommand{\tp}[3]{\{#1 , #2 , #3\}}
\newcommand{\trop}[3]{[#1 , #2 , #3]}
\newcommand{\utro}[1]{\alpha_{#1}}

\newcommand{\jp}{\circ}
\newcommand{\OP}{\mathrm{op}}

\newcommand{\tmpref}[1]{}

\newcommand{\JCstar}{$JC^*$}
\newcommand{\Cstar}{$C^*$}
\newcommand{\Wstar}{$W^*$}
\newcommand{\JC}{$JC$}
\newcommand{\JW}{$JW$}
\newcommand{\JBstar}{$JB^*$}
\newcommand{\JWstar}{$JW^*$}

\newcommand{\BARotimes}{\mathbin{\overline{\otimes}}}

\newcommand{\cvtrace}{\tau}

\def\simone{\mathinner{\mkern7mu\raise-5pt
        \vbox{\kern7pt\hbox{{\tiny 1}}}\mkern-10mu
	        \hbox{$\sim$}\mkern2mu}
}

\numberwithin{equation}{section}
\newtheorem{proposition}{Proposition}[section]
\newtheorem{lemma}[proposition]{Lemma}

\newtheorem{theorem}[proposition]{Theorem}
\newtheorem{corollary}[proposition]{Corollary}
\theoremstyle{definition}
\newtheorem{remark}[proposition]{Remark}
\newtheorem{remarks}[proposition]{Remarks}

\newtheorem{definition}[proposition]{Definition}

\begin{document}
\title{Universally reversible \JCstar-triples and operator spaces}
\author[L.~J.~Bunce]{Leslie J. Bunce}
\address{Department of Mathematics, University of Reading, Reading RG6
2AX}
\email{l.j.bunce@reading.ac.uk}
\author[R.~M.~Timoney]{Richard M. Timoney}
\address{School of Mathematics\\Trinity College\\Dublin 2}
\email{richardt@maths.tcd.ie}
\thanks{The work of the second author was supported
by the Science Foundation Ireland under grant 11/RFP/MTH3187.}
\date{Version \today}



\begin{abstract}
We prove that the vast majority of \JCstar-triples satisfy the condition
of universal reversibility.
Our characterisation
is that a \JCstar-triple is universally reversible if and only if it has
no triple homomorphisms onto Hilbert spaces of dimension greater than
two nor onto spin factors of dimension greater than four. We establish
corresponding characterisations in the cases of \JWstar-triples and of
TROs (regarded as \JCstar-triples). We show that the distinct natural
operator space structures on a universally reversible \JCstar-triple
$E$ are
in bijective correspondence with a distinguished class of ideals in
its universal TRO, identify the Shilov boundaries of these operator
spaces and prove that $E$ has a unique natural operator space structure
precisely when $E$ contains no ideal isometric to a nonabelian TRO.
We deduce some decomposition and completely contractive properties of
triple homomorphisms on TROs.
\end{abstract}
\maketitle

\section{Introduction}

The norm closed subspaces of $\BH$ invariant under the ternary
product  $\trop a b c  =  ab^*c$
and known as TROs (ternary rings of operators)
have a well-documented significance in the category of operator spaces in
which, up to complete isometry, they occur as the noncommutative Shilov
boundaries, injective envelopes and as  Hilbert \Cstar-modules
\cite[Chapters 4, 8]{BlecherleMerdy}.
TROs have a natural operator space structure since the
algebraic isomorphisms between them are exactly the surjective complete
isometries. This is a state generally not possessed by the norm closed
subspaces of TROs invariant under the symmetrised triple product 
$\tp a b c  =  (1/2)(\trop abc + \trop cba)$
which arose in a different tradition
and were named $J^*$-algebras by Harris
\cite{Harris74,Harris81}
who showed that the
open unit ball of each such space is a bounded symmetric domain, that
is, the group of biholomorphic automorphisms of the open unit ball is
transitive. Kaup's extension \cite{Kaup}
is a formidable conjunction of
algebra and analysis: the open unit ball of a complex Banach space $E$ is
a bounded symmetric domain if and only if there is a continuous map 
$\tp \cdot \cdot \cdot \colon E^3 \to E$
such that the operator $D(a, b)$ on $E$ given by $D(a, b)(c)
=  \tp a b c$ is sesquilinear, satisfies the \emph{Jordan triple identity}, 
$[D(a, b), D(x, y)]  =  D( \tp a b x, y)  -  D(x, \tp y a b )$
and that $D(x, x)$
is a positive hermitian operator on $E$ with norm $\|x\|^2$
(for all $a, b, c, x$ and $y$ in $E$). This class of Banach spaces, the 
\emph{\JBstar-triples}, is invariant
under linear isometries and a key feature is that the \emph{Jordan triple product}
$\tp a b c$ on a \JBstar-triple is unique: the triple isomorphisms between
\JBstar-triples are the surjective linear isometries.

   The linear isometric copies of $J^*$-algebras, hereafter referred
to as \JCstar-triples, are principal examples of \JBstar-triples. In these
terms $J^*$-algebras are precisely the \JCstar-subtriples of $\BH$, or \emph{concrete}
\JCstar-triples, and by the triple
Gelfand-Naimark
theorem of Friedman
and Russo \cite{FRgn}
\JCstar-triples are the \JBstar-triples with vanishing
‘’exceptional ideal’’. The \JCstar-triples with a predual are
called \emph{\JWstar-triples}. Hilbert spaces are \JCstar-triples as are all (linear
isometric copies of) Cartan factors, \Cstar-algebras, TROs and Jordan operator
algebras. Friedman and Russo \cite{FRJFA1985}
have also shown that the range of
a contractive projection on a \Cstar-algebra is a \JCstar-triple, though not
necessarily a subtriple of the ambient \Cstar-algebra (one by-product of
this paper is the passing observation
that all simple \JCstar-triples arise in this way up to linear
isometry).

  Operator space structure of reflexive \JCstar-triples has been
investigated in important articles
by Neal, Ricard and Russo \cite{NealRicardRusso06} and
by Neal and Russo \cite{NealRussoTAMS03,NealRussoPAMS06}, who also
proved \cite{NealRussoPacJ03} that an operator space $X$
is completely isometric to a TRO
if and only if $M_n(X)$ is a \JCstar-triple (in the abstract sense
defined above) for all $n \geq 2$.
In \cite{BunceFeelyTimoney}
B. Feely and the authors began
a general study of \emph{\JC-operator spaces} (the operator spaces induced
by linear isometries onto concrete \JCstar-triples) which was continued in
\cite{BunceTimoneyII,BunceTimoney3}. The \JC-operator spaces of all
Cartan factors (see \S\ref{sec:preliminaries}) were described
and enumerated in the process  via instrumental use of the notions
(inaugurated in  \cite{BunceFeelyTimoney})  of the universal TRO of a \JCstar-triple and of
a universally reversible \JCstar-triple, conceived by analogy with
companion notions in Jordan operator 
algebras \cite{HancheOlsenJC,HOS} which they precisely generalise
\cite[\S4]{BunceFeelyTimoney}.

Universally reversible \JCstar-triples form a class of Banach spaces
preserved by linear isometries.
The extent and influence of this class is the subject of this paper.

The setting is that given
$a_1$, \dots , $a_{2n+1}$ in a \JCstar-subtriple $E$ of $\BH$ with
$n \geq 2$, the \emph{reversible} element
\[
a_1a_2^*a_3 \cdots a_{2n}^*a_{2n+1}  + a_{2n+1} a_{2n}^* \cdots  a_3a_2^*a_1
\]
(lying in the TRO generated by $E$) is not a Jordan-theoretic product and
has no compelling algebraic claim to membership of
$E$: if $E$ contains all reversible elements arising in this way it is
said to be \emph{reversible} in $\BH$.
A \JCstar-triple E is defined to be
\emph{universally reversible} if $\pi(E)$ is reversible in
$\BH$ for every triple homomorphism $\pi \colon E \to \BH$.

By \cite{BunceFeelyTimoney}
Cartan factors are universally reversible with the 
exceptions of Hilbert spaces of dimension $\geq 3$  and spin factors 
of dimension $\geq  5$. We shall prove that the latter are essentially 
the only obstacles to universal reversibility thus showing 
that most \JCstar-triples satisfy the condition and that failure 
to do so is confined to a sharply delineated isolated class.

Our characterisation is that a \JCstar-triple is not universally 
reversible when, and only when, it has a triple homomorphism 
onto a Hilbert space of dimension $\geq 3$ or a spin factor of 
dimension $\geq  5$ (Theorem~\ref{thm:charaterise}).
Further results in \S\ref{sec:characterising}
give the structures of 
universally reversible \JWstar-triples (in
Theorem~\ref{thm:jwstarurchar}) and \Wstar-TROs
(Corollary~\ref{corol:wstartroUR}), and that section
contains other results of independent 
interest. The prior section, \S\ref{sec:Reversibility},
deals with important special 
cases such as one-sided weak*-closed ideals in von Neumann 
algebras.  Having shown the prevalence of universally reversible 
\JCstar-triples we study their \JC-operator spaces in
\S\ref{sec:JCOpspaces}, distilling 
the significance (as it turns out) of nonabelian TROs in the 
process (Theorem~\ref{thm:20120622512}).
For a given universally reversible \JCstar-triple E we 
prove that its \JC-operator spaces are in bijective correspondence 
with a distinguished family of ideals in its universal 
TRO (Theorem~\ref{thm:201206510}),
identify the Shilov boundaries of each such operator space 
(Corollary~\ref{coroll:ShilovBdry})
and severally characterise those E with a unique \JC-operator 
space structure (Theorem~\ref{thm:20120622512}).
Modulo a mild restriction, we apply our results 
to show that triple homomorphisms on a TRO routinely decompose 
into the sum of a TRO homomorphism and a TRO antihomomorphism 
(Proposition~\ref{prop:triplehomdecomp})
and deduce that a triple automorphism on a \Wstar-TRO factor
not isometric 
to a von Neumann algebra must be a TRO automorphism
(Theorem~\ref{thm:WstarFactorCI}).

\section{Preliminaries}
\label{sec:preliminaries}

We refer to surveys
\cite{RodSurvey1994,RussoSurvey}
and articles
\cite{EdwardsRuttimann1988JLMS,FRgn,Harris74,Harris81,KaupFibreBund}
for general background on \JCstar-triples
and to \cite{BlecherleMerdy,ER}
for the theory of operator spaces and also for TROs, further
information about which may be found in
\cite{BunceFeelyTimoney,BunceTimoney3,EffrosOzawaRuan,Hamana99,Zettl}.
Related to \JCstar-triples are the \JCstar-algebras, the
complexifications of the \JC-algebras studied thoroughly in
\cite{HOS}. \JCstar-algebras are the norm closed subspaces of
\Cstar-algebras invariant under the involution and Jordan product 
$a \circ b = (ab + ba)/2$, which are all \JCstar-triples because
$\tp a b c = (a \circ b^*) \circ c + a \circ (b^* \circ c) - (a \circ
c) \circ b^*$.
A \JWstar-algebra is a weakly closed \JCstar-subalgebra of a von
Neumann algebra.
Given elements $a_1, \ldots, a_{2n+1}$ of a TRO $T$, we often write
$\trop{a_1}{\cdots}{a_{2n+1}}$ for the element
$a_1 a_2^*a_3 \cdots a_{2n}^*a_{2n+1}$ of $T$.

Throughout this paper the term ideal shall mean norm closed ideal. In
each of the four above mentioned categories ideals have the obvious
algebraic
definition with respect to the relevant product and, by a result
of Harris        \cite[Proposition 5.8]{Harris81},
the definitions coincide
whenever the categories overlap. We write $\TRO(X)$ for the TRO generated
by subset  $X$ in a \Cstar-algebra and recall \cite[Proposition
3.2]{BunceTimoney3} that
$\TRO(I)$ is an ideal of $\TRO(E)$ with $I  =  E \cap∩ \TRO(I)$
whenever $I$ is
an ideal of a concrete \JCstar-triple $E$.

Let $E$ be a \JCstar-triple. Elements $x$ and $y$ in $E$
are said to be orthogonal
(denoted by  $x \perp y$) if $\tp xxy  =  0$, which is equivalent to 
$x^*y  = xy^*  =  0$
when $E$ is a \JCstar-subtriple of $\BH$, implying that 
$\|x  + y\|  =  \max(\|x\|, \|y\|)$ \cite[p. 18]{Harris74}.
Given $x \in E$ and  $S \subseteq  E$ we
write $x \perp S$ if $x \perp y$ for all $y$ in $S$ and denote 
$\{x \in E : x \perp S \}$ by $S^\perp$.
Ideals $I$ and $J$ of $E$ are orthogonal, that is, $I \subseteq
J^\perp$ ⊥  (equivalently, $J \subseteq I^\perp$)
if and only if $I \cap J  =  \{0\}$,
in which case, $I + J  =  I \oplus_\infty ∞ J$.  If $E$ is a 
\JWstar-triple with a
weak*-closed ideal $I$ then $E$ is the sum of $I$ and its complementary
weak*-closed ideal  $I^\perp$.   A \JWstar-triple $E$ is a factor if it has
no nontrivial weak*-closed ideals. Given $x$ in $E$, $Q_x$  denotes the
conjugate linear operator $y \mapsto  \tp x y x$ on $E$.

An element $u$ of a \JCstar-triple $E$ is a \emph{tripotent}
if $\tp u u u = u$. The \emph{Peirce projections} $P_k(u)$ ($k = 0,
1,2$, $u$ a tripotent) on $E$ associated with $u$ are given by
\[
P_2(u) = Q_u^2,
P_1(u) = 2(D(u,u) - P_2(u))
\mbox{ and }
P_0(u) = I - P_2(u) - P_0(u)
\]
and satisfy
$P_i(u)P_j(u) = 0$ whenever $i \neq j$. Also
$E = E_2(u) \oplus E_1(u) \oplus E_0(u)$ [linear direct sum]
where $E_k(u) = P_k(u)E = \{ x \in E: 2 \tp u u x = k x\}$ is
the \emph{Peirce $k$-space} for $u$ and each $E_k(u)$ is a 
\JCstar-subtriple of $E$. In addition,
$E_2(u)$ is a \JCstar-algebra with 
product
$x \jp y = \tp x u y$, identity element $u$ and involution
$x^{\#} = \tp u x u $.
If $E$ is a \JWstar-triple then $E_2(u)$ is a \JWstar-algebra.

A tripotent $u \in E$ is said to be \emph{abelian} if $E_2(u)$ is an
abelian \JCstar-algebra (equivalently, an abelian \Cstar-algebra)
and to be \emph{minimal} if $E_2(u) = \IC u$.
If $P_0(u) =0$, then $u$ is called a \emph{complete}
tripotent of $E$ and is called a \emph{unitary} tripotent if
$P_0(u) = P_1(u) = 0$.
The completeness of a tripotent $u$ is equivalent to
$u^\perp = \{0\}$
(by \cite[p. 18]{Harris74}).
and to
$u$ being an extreme
point of the closed unit ball of $E$
(by \cite[Theorem 11]{Harris74}).

We note the following.

\begin{lemma}
\label{lem:tripots}
Let $u$ be a tripotent in a \JCstar-triple $E$, $M$ a \JWstar-triple,
and let $\pi \colon E \to M$ be a nonzero triple homomorphism
with weak*-dense range.
If $u$ is a complete (respectively, a unitary, an abelian, a minimal)
tripotent of $E$, then
$\pi(u)$ is a complete (respectively, a unitary, an abelian, a minimal)
tripotent of $M$. In
particular $\pi(u) \neq 0$ if $u$ is a complete tripotent of $E$.
\end{lemma}

By a \emph{Cartan factor} we shall mean a \JCstar-triple
linearly isometric to one of the following \emph{canonical forms}
(i) $\BH e$;
(ii) $\{ x \in \BH : x^t = x\}$, $\dim H \geq 2$;
(iii) $\{ x \in \BH : x^t = -x\}$, $\dim H \geq 4$;
(iv) a spin factor, where in (i) -- (iii)
$H$ is a Hilbert space,
$e \in \BH$ is a projection
and $x \mapsto x^t$ is a
transposition.
Up to linear isometry, the forms (ii), (iii) when $H$ has even or
infinite dimension, and (iv), are type I \JWstar-triples and those of
form (i) where $e$ is a minimal projection are Hilbert spaces.
The spin factor of finite dimension $n+1$ is denoted by $V_n$.

We denote the norm closed linear span of the minimal tripotents in a
Cartan factor $C$ by $K_C$, 
an
ideal of $C$ (which consists
of the compact operators contained in $C$ when $C$ is of type (i),
(ii) or (iii) above).
The Cartan factors may be characterised as the \JWstar-triple factors
possessing minimal tripotents.
The \emph{rank} of a Cartan factor $C$ is the cardinality of a maximal
family of orthogonal minimal tripotents in $C$.

The following three conditions are equivalent for a Cartan factor $C$:
\begin{equation}
\mbox{(a) $C$ has finite rank;   (b)
$K_C  =  C$;   (c) $C$ is reflexive.}
\end{equation}

A \emph{Cartan factor representation} of a \JCstar-triple $E$ is a
triple homomorphism $\pi \colon E \to C$  where $C$ is a Cartan factor
with $\pi(E)$ weak*-dense in $C$ (so that $\pi(E) = C$ if $C$ has
finite rank).
If $\pi \colon E \to C$ is a 
a Cartan factor representation, the weak*-continuous extension
$\tilde{\pi} \colon E^{**} \to C$ is a surjective triple homomorphism
so that the orthogonal complement of $\ker \tilde{\pi}$ is a
weak*-closed ideal of $E^{**}$ linearly isometric to $C$.
Conversely, if $J$ is a weak*-closed ideal of $E^{**}$ linearly
isometric to $C$, then the restriction to $E$ of the natural
projection from $E^{**}$ onto $J$ induces a Cartan factor
representation $\pi \colon E \to C$.
Every \JCstar-triple $E$ has a separating family of Cartan factor
representations \cite{FRgn}.

The type of a \JWstar-algebra is the type of its self-adjoint part,
a \JW-algebra, and the corresponding decomposition theory of
\JWstar-algebras may be read directly from
\cite{HOS,Stacey1981,Stacey1982,Topping}.
In particular, a \JWstar-algebra is of type I if it is the norm
closed linear span of abelian projections and is continuous if it
has no nonzero abelian projections. We briefly recall the extension
to \JWstar-triples due to Horn and Neher
\cite{HornPredual,HornTypeI,HornNeher}.

Following Horn \cite{HornTypeI} we use $B ⨂ \BARotimes  N$ to denote the weak* closure
of the algebraic tensor product  $B \otimes N$ in the von Neumann algebra
tensor product  $\BH \BARotimes ⨂ \mathcal{B}(K)$
when $B$ is an abelian von Neumann
subalgebra of $\BH$ and N is a \JWstar-subtriple of $\mathcal{B}(K)$.

A \JWstar-triple is defined to be of \emph{type I} if it is the
weak*-closed linear span of its abelian tripotents, and to be
\emph{continuous} if it contains no nonzero abelian tripotents
\cite{HornPredual,HornNeher}. By Horn's theorem \cite{HornTypeI}, a
\JWstar-triple $M$ is type I if and only if $M$ is linearly isometric
to an $\ell^\infty$ sum
\begin{equation}
\label{eqn:TypeI}
\bigoplus_j B_j \BARotimes C_j
\end{equation}
where each $B_j$ is an abelian von Neumann algebra and each $C_j$ is a
Cartan factor in canonical form. When all $B_j$ are nonzero and the
$C_j$ are mutually distinct we shall refer to $M$ as being
\emph{$\{ C_j\}$-homogeneous}.

We further say that a \JWstar-triple
$M$ is of \emph{type I$_{\mathrm{finite}}$}
if each $C_j$ has finite rank in the decomposition (\ref{eqn:TypeI}).
Specialising, we refer to $M$ as \emph{type I$_1$} if each
$C_j$ is (isometric to) a Hilbert space, and as 
\emph{type I$_{\mathrm{spin}}$} if each $C_j$ is a spin factor.
Equivalently (by \cite[\S2]{HornTypeI}) $M$ is type I$_1$ if and only
if
it contains a complete abelian tripotent and (by \cite[Theorem
6.3.14]{HOS}) $M$ is type I$_{\mathrm{spin}}$
if and only if it is linearly isometric to a \JWstar-algebra
with self-adjoint part of
type I$_2$ in the sense of \cite[5.3.3]{HOS}.
If each $C_j$ has infinite rank we say that $M$ has type I$_\infty$,
in which case, collecting terms, with $\cong$ meaning `linearly
isometric to', we have
\begin{equation}
\label{eqn:typeIinfinity}
M \cong N \oplus_\infty We,
\end{equation}
where $N$ is a type I$_\infty$ \JWstar-algebra and $e$ is a properly
infinite projection in a type I$_\infty$ von Neumann algebra $W$.

In general, by \cite[(1.17)]{HornNeher} every \JWstar-triple $M$ has a
decomposition
\begin{equation}
\label{eqn:typeIpluscontinuous}
M \cong M_1 \oplus_\infty M_c,
\end{equation}
where the first summand is type I and the second is continuous,
referred to as the type I and continuous parts of $M$, respectively.
In addition (see \cite[\S4]{HornNeher})
\begin{equation}
\label{eqn:continuousdecomp}
M_c \cong N \oplus_\infty We,
\end{equation}
where $N$ is a continuous \JWstar-algebra and $e$ is a projection in a
continuous von Neumann algebra $W$.

We remark that a \JWstar-algebra is continuous  as a \JWstar-algebra if and
only if it is continuous as a \JWstar-triple, the `if' part being
clear since abelian projections are abelian tripotents. Conversely,
suppose $M$ is a continuous \JWstar-algebra. If $M$ is not a continuous
\JWstar-triple, then the decomposition theory implies a surjective
linear isometry
$\pi \colon N \to→ B \BARotimes  ̅ C$, for some weak*-closed ideal $N$ (a continuous
\JWstar-algebra) of $M$, where $B$ is an abelian von Neumann algebra and $C$
is a Cartan factor in canonical form possessing a unitary tripotent,
$u$, say. In which case, with $v  =  1 \otimes u$ and $w  =  \pi^{-1}(v)$, we arrive
at the contradiction that the type I \JWstar-algebra $B \BARotimes  ̅ C_2(u)$
is Jordan*-isomorphic to $N_2(w)$ which, by
\cite[(5.2) Lemma]{HornNeher},
is a
continuous \JWstar-algebra. It follows that the two forms of `type
I' for \JWstar-algebras also coincide.

We have defined universally reversible \JCstar-triples in the
Introduction. By \cite[\S4]{BunceFeelyTimoney},
when $A$ is
a \JCstar-algebra, $A$ is universally reversible if and only if
$\pi(A)_{sa} $ is a reversible \JC-algebra
(see \cite[2.3.2]{HOS}) for each Jordan *-homomorphism $\pi \colon A
\to \BH$.

It follows from these remarks together with the coordinatisation
theorem \cite[2.8.9]{HOS} that a unital \JCstar-algebra is a
universally reversible \JCstar-triple if it contains, for $3 \leq n
< \infty$, $n$ orthogonal projections with sum 1 and pairwise
exchanged by symmetries.
The following is recorded for later use.

\begin{proposition}
\label{prop:revjcstaralg}
Let $M$ be a \JWstar-algebra such that $M$ is continuous or of type
I$_\infty$ or of type I$_n$ with $3 \leq n < \infty$. Then
$M$ is a universally reversible \JCstar-triple.
\end{proposition}

\begin{proof}
$M$ contains $m$ orthogonal projections with sum 1 pairwise exchanged
by symmetries with $m = 4$
if $M$ is continuous or of type I$_\infty$
(see the proof of \cite[Theorem 5.3.9]{HOS}),
and with $m=n$
if $M$ has type I$_n$ for $3 \leq n < \infty$.
\end{proof}

The \emph{universal TRO} \cite{BunceFeelyTimoney}
of a \JCstar-triple $E$ is a pair $(\TROu(E), \utro E)$
consisting of an injective triple homomorphism 
\[
\utro E \colon E \to \TROu(E)
\]
where
$\TROu(E)$ is a TRO generated  by $\utro E(E)$
(as a TRO)
and possessing the universal property that for each triple homomorphism
$\pi \colon E \to \BH$ there is a (unique)
TRO homomorphism $\tilde{\pi} \colon
\TROu(E) \to \BH$ with $\tilde{\pi} \circ \utro E = \pi$.
The \emph{canonical involution} $\Phi \colon \TROu(E) \to \TROu(E)$,
which has order 2, is the unique antiautomorphism of $\TROu(E)$ fixing
each element of $\utro E (E)$.
By \cite{BunceFeelyTimoney}, a \JCstar-triple $E$ is universally
reversible if and only if $\utro E (E) = \{ x \in \TROu(E) : \Phi(x) =
x \}$.
When $A$ is a \JCstar-algebra $(\TROu(A), \utro A)$ coincides with the
universal \Cstar-algebra 
$ (C^*_{\mathrm{J}}(A), \beta_A)$
of $A$ 
\cite[Proposition 3.7]{BunceFeelyTimoney}.

\begin{proposition}
\label{prop:TROimage}
Let $A$ be a \JCstar-algebra and $T$ a TRO and let $\pi \colon A \to
T$ be a surjective linear isometry. Then there is a \Cstar-algebra
$B$, a Jordan *-isomorphism $\phi \colon A \to B$ and a TRO isomorphism
$\psi \colon B \to T$ such that $\pi = \psi \circ \phi$.
\end{proposition}

\begin{proof}
We may suppose that $A \subseteq C^*_{\mathrm{J}}(A)$. By
\cite[Proposition 3.7]{BunceFeelyTimoney} there is a TRO homomorphism
$\tilde{\pi} \colon  C^*_{\mathrm{J}}(A) \to T$ extending $\pi$.
Putting $I = \ker \tilde{\pi}$ we have the commutative diagram
\[
\xymatrix{
A
\ar@{->}[r]^{\pi}
\ar@{^{(}->}[d]
& 
T\\
C^*_{\mathrm{J}}(A) \ar@{->}[ur]^{\tilde{\pi}} \ar@{->}[r]^{q}&
C^*_{\mathrm{J}}(A)/I
\ar@{->}[u]_{\tilde{\pi}_I}
}
\]
where $q$ is the quotient map and $\tilde{\pi}_I$ is the map sending
$x + I$ to $\pi(x)$ for each $x$ in $C^*_{\mathrm{J}}(A)$.
Since
$\pi(A)= T$, $\tilde{\pi}_I$ is a TRO isomorphism
and
since
$\tilde{\pi}_I \circ q$ agrees with $\pi$ on $A$ we have $q(A)
= C^*_{\mathrm{J}}(A)/I$. The assertion follows upon setting $B =
C^*_{\mathrm{J}}(A)/I$, $\psi = \tilde{\pi}_I$ and $\phi$ to be the
restriction of $q$ to $A$.
\end{proof}

In outline an \emph{operator space} is a complex Banach space $E$ together with
a linear isometric embedding into $\BH$ and the resulting operator space
structure on $E$ is determined by the matrix norms  on $M_n(E)$  conferred
thus by the \Cstar-algebras $M_n(\BH)$, for all $n \in ɛ \IN$. A linear map 
$\pi \colon E \to F$
between operator spaces is said to be \emph{completely contractive}
if, for each $n$, $\|\pi_ n\| \leq  1$
where $\pi_n$  is the tensored map $\pi \otimes I_n
\colon M_n(E) \to→ M_n(F)$, and to be \emph{completely isometric}
if each $\pi_n$ is a
surjective isometry. Complete isometries are the isomorphisms in the
category of operator spaces. By a \JC-operator space structure on
a Banach space $E$ we mean an operator space structure determined by a
linear isometry from $E$ onto a \JCstar-subtriple of $\BH$ (in which case,
$E$ is a \JCstar-triple), and by a \JC-operator space we mean a \JCstar-triple $E$
together with a prescribed \JC-operator space structure on E.
By \cite[Proposition 6.2]{BunceFeelyTimoney} the possible \JC-operator
space structures on a \JCstar-triple $E$ arise from
(norm closed) ideals $\mathcal{I}$
of $\TROu(E)$ for which $\utro E(E) \cap \mathcal{I} = \{0\}$, called
\emph{operator space ideals} of $\TROu(E)$.
For each operator space ideal $\mathcal{I}$
of $\TROu(E)$ we have the \JC-operator space
structure,
$E_{\mathcal{I}}$,
on $E$ determined by the isometric embedding $E \to
\TROu(E)/\mathcal{I}$ ($x \mapsto \utro E (x) + \mathcal{I}$).

Given an ideal $I$ of a \JCstar-triple $E$ we recall
\cite[Theorem 3.3]{BunceTimoney3}
that $\TRO(\utro E (I))$ is an ideal of
$\TROu(E)$,
$(\TROu(I), \utro I ) = ( \TRO(\utro E (I)), \utro E |_I )$
and $\TROu(E/I) = \TROu(E)/\TROu(I)$ with
$\utro {E/I}(x+I) = \utro E (x) + \TROu(I)$ for $x \in E$.

\begin{lemma}
\label{lem:opspidealsofideals}
Let $I$ be an ideal of a \JCstar-triple $E$ and $\mathcal{J}$ an ideal
of $\TROu(E)$. Then  $\mathcal{J}$ is an operator space ideal of
$\TROu(I)$
if and only if $\mathcal{J}$ is an operator space ideal of $\TROu(E)$
contained in $\TROu(I)$; in which case $I_{\mathcal{J}}$ is an
operator subspace  of $E_{\mathcal{J}}$.

Thus, $E$ has at least as many distinct \JC-operator space structures
as does $I$.
\end{lemma}

\begin{proof}
We have
\[
\mathcal{J} \cap \TROu(I) \cap \utro E (I)
= \mathcal{J} \cap \utro E (I)
= \mathcal{J} \cap \utro I (I)
\]
and an ideal of $\TROu(I)$ is an ideal of $\TROu(E)$ (see 
\cite[Proposition 3.2]{BunceTimoney3} for example),
from which the assertions follow.
\end{proof}

\section{Reversibility for \JWstar-triples}
\label{sec:Reversibility}
\numberwithin{equation}{section}

The aim of this section is to show that all continuous and type
I$_\infty$ \JWstar-triples are universally reversible
(Theorem~\ref{thm:URifNoTypIfin})
and to use this
to establish reversibility criteria for \JWstar-triples.
In view of the Horn-Neher structure theory
(\ref{eqn:typeIinfinity})
and
(\ref{eqn:continuousdecomp}),
and Proposition~\ref{prop:revjcstaralg}
(and \cite[Proposition 3.6]{BunceFeelyTimoney} which allows passage to
direct sums of finitely many terms)
we may confine our attention to \JWstar-triples of the
form $eW$ where $e$ is a projection in a von Neumann algebra $W$.
Moreover it suffices to deal with the cases where (a) $W$ is
continuous (Propositions \ref{ContIsRev} and
\ref{PropPIG}) and (b) $eWe$ is type I$_\infty$
(Proposition~\ref{PropPIG}).

Given projections $e$ and $f$ in a von Neumann algebra $W$ we write $e
\sim f$ if and only if $e = u u^*$ and $f = u^*u$ for some (partial
isometry) $u \in W$, and $e \precsim f$ if and only if there is a
projection $q$ in $W$ with $e \sim q \leq f$.
In addition for $n \in \IN$ we use the notation
$n \cdot e \precsim f$ to mean the existence of
orthogonal projections $e_1, \ldots , e_n \in W$
with $\sum_{i=1}^n e_i \precsim f$ and $e \sim e_i$ for $i= 1 ,
\ldots, n$.
We note that if $e, f,p$ are projections in $W$ with $e, f \leq p$
and $e \sim f$ in $W$ then $e \sim f$ in $pWp$ since if $u \in W$ with
$e = uu^*$ and $f = u^*u$ the $u = euf \in pWp$.

We shall make frequent use of the properties 
\cite[p. 291]{TakesakiVolI}
that, in a von Neumann algebra, if $e$ and $f$ are projections with 
$e \precsim f$
and
$f \precsim e$
then $e \sim f$, and that if
 $(e_i )_{ i \in I}$
 and
 $(f_i )_{ i \in I}$
 are two families of orthogonal projections
with
$e_i \sim f_i$ for each $i \in I$, then $\sum_{i \in I} e_i \sim \sum_{i
\in I} f_i$.
We refer to
\cite[Chapter V]{TakesakiVolI}
for any undefined terms used below.

\begin{lemma}
\label{SquareRev}
Let $e$ and $f$ be projections in a von Neumann algebra $W$ such that
$e \sim f$. Then $eWf$ is a universally reversible \JCstar -triple.
\end{lemma}

\begin{proof}
Given $u \in W$ with $e = u^*u$ and $f = uu^*$, the map $x \mapsto ux$
is a TRO isomorphism
from $eWf$ onto the von Neumann algebra $fWf$,
which  is universally reversible as a \JCstar-triple
by \cite[Lemma 3.4 and Theorem 4.6]{AlfsenHOSchultzActa1980}
(together with \cite[\S4]{BunceFeelyTimoney}).
\end{proof}

The following folklore is included for want of an exact
reference.

\begin{lemma}
\label{LemPID}
Let $e$  and $f$ be projections in 
a von-Neumann algebra with $e
\precsim f$. Then there is a projection $p \in W$ with $e \leq p \sim f$.
\end{lemma}

\begin{proof}
Let $q$ be a projection in $W$ such that
$e \sim q \leq f$. By the
comparability theorem 
(\cite[p. 293]{TakesakiVolI}
there is a central
projection
$z$ of $W$ such that $(1-e)z \precsim (f-q)z$ and $(f-q)(1-z)
\precsim
(1-e)(1-z)$.
Thus, using $ez \sim qz$,
\[
z = (1-e)z + ez \precsim (f-q)z + qz = fz \leq z,
\]
giving $z \sim fz$.

Since there is a projection $r$ in $W$ with
$(f-q)(1-z) \sim r \leq (1-e)(1-z)$,
so that
\[
f(1-z) = (f-q)(1-z) + q(1-z) \sim r +e(1-z)
\]
we have
\[
f = fz + f(1-z) \sim z + (r +e(1-z)) = e + (1-e)z + r \geq e,
\]
as required.
\end{proof}

\begin{lemma}
\label{subdivideTripotents}
Let $e$ be a projection in a continuous von Neumann algebra $W$, let
$u$ be a tripotent in $eW$, and
let $n \in \IN$.
Then there exist orthogonal
tripotents $u_1, \ldots, u_n \in eW$ with $u = \sum_{i=1}^n u_i$
and $n \cdot (u_i^* u_i) \precsim e$ for $i = 1, \ldots, n$.
\end{lemma}

\begin{proof}
Let $f = u u^*$. We may choose orthogonal projections $f_1,
\ldots, f_n$ in $W$ with
$f =
\sum_{i=1}^n f_i$ with $f_i \sim f$ for $i = 1, \ldots, n$
(see \cite[6.5.6]{KadisonRingroseII}, for instance).
Now letting $u_i = f_i u$ for $i = 1, \ldots, n$,
we have that $u_1, \ldots, u_n$
are orthogonal tripotents
with
$u = u_1 + \cdots + u_n$ and
$u_i^* u_i \sim u_i u_i^* = f_i \leq e$
for $i = 1, \ldots, n$
giving  $n \cdot (u_i^* u_i) \precsim e$.
\end{proof}

\begin{lemma}
\label{SupOfNSmallProjections}
Let $e$ be a finite projection in a continuous von Neumann algebra
$W$ and let $f_1, f_2, \ldots, f_n$ be projections in $W$ with $n
\cdot f_i \precsim e$ for $i = 1, \ldots, n$. Then $\bigvee_{i=1}^n
f_i \precsim e$.
\end{lemma}

\begin{proof}
Let $f = \bigvee_{i=1}^n
f_i$ and $g = e \vee f$
(finite by \cite[Theorem V.1.37]{TakesakiVolI})
and let $\cvtrace$ be the centre-valued trace
on the finite von Neumann algebra
$gWg$.
Using
\cite[Corollary V.2.8]{TakesakiVolI} we have
$\cvtrace(f_i) \leq (1/n) \cvtrace(e)$ for $i = 1, \ldots, n$
so that
$\cvtrace(f) \leq \sum_{i=1}^n \cvtrace(f_i) \leq \cvtrace(e)$ and
hence $f \precsim e$.
\end{proof}

\begin{lemma}
\label{lem:ntripotents}
Let $u_1, \ldots, u_n$ be tripotents in $eW$ where $e$ is a finite
projection in
a  continuous von Neumann algebra
$W$
such that $n \cdot (u_i^* u_i) \precsim e$
 for $1 \leq i \leq n$.
Then there is a universally reversible \JCstar-subtriple of $eW$
containing $u_1, \ldots, u_n$.
\end{lemma}

\begin{proof}
By 
Lemma~\ref{SupOfNSmallProjections}
together with
Lemma~\ref{LemPID}
there is a projection $f$ in $W$ with
$\bigvee_{i=1}^n u_i^* u_i \leq f \sim e$
so that each
$u_i \in e Wf$ which, by 
Lemma~\ref{SquareRev}, is a universally reversible \JCstar-triple.
\end{proof}

\begin{proposition}
\label{ContIsRev}
Let $e$ be a finite projection in a continuous von Neumann algebra
$W$. Then $eW$ is a universally reversible \JCstar -triple.
\end{proposition}

\begin{proof}
Let $\pi \colon eW \to \BH$
be a triple homomorphism
and let
$a_1, \ldots, a_n \in \pi(eW)$ for an odd integer $n \geq 3$.
We must show that
\begin{equation}
\label{eqn10REV}
\trop{a_1}{\cdots}{a_n} +
\trop{a_n}{\cdots}{a_1} 
\in \pi(eW).
\end{equation}
Choose $x_i \in eW$ such that $\pi(x_i) = a_i$, for $1 \leq i \leq
n$.
In order to establish
(\ref{eqn10REV}),
since $eW$
is the norm closed linear span of its tripotents,
we may suppose without loss
that each $x_i$ is a tripotent.
In which case,
by Lemma~\ref{subdivideTripotents},
we may write $x_i = \sum_{j=1}^n u_{i,j}$ where the $u_{i,j}$ are
tripotents in $eW$ with
$n \cdot (u_{i,j}^* u_{i,j}) \precsim e$
for $1 \leq i, j \leq n$. Put $a_{i,j} = \pi(u_{i,j})$ for $1 \leq i,
j \leq n$.
By Lemma~\ref{SquareRev},
for each choice of $j_1, \ldots, j_n$ in $\{ 1, \ldots, n\}$ there is
a universally reversible \JCstar-subtriple $M$ of $eW$ containing
$u_{1, j_1}, \ldots, u_{n, j_n}$ from which it follows that
\[
\trop{a_{1, j_1}}{\cdots}{a_{n,j_n}} +
\trop{a_{n, j_n}}{\cdots}{a_{1, j_1}} 
\in \pi(M) \subseteq \pi(eW).
\]
By summing over all such choices of $j_1, \ldots, j_n$ 
we obtain (\ref{eqn10REV}).
\end{proof}

We next turn to the more pliable properly infinite projections $e$ in
a von Neumann algebra $W$.
In this case we recall
\cite[Proposition V.1.36]{TakesakiVolI}
there are projections $p$ and $q$ in $W$ with $e = p+q \sim p \sim q$,
so that if $f_1$ and $f_2$ are projections with
$f_1 \precsim e$ 
and
$f_2 \precsim e$ 
then
$f_1 \vee f_1 - f_2 \sim f_1 - f_1 \wedge f_2 \precsim p$ and $f_2
\precsim q$ giving
$f_1 \vee f_2 \precsim p+q =e$.
The next assertion is an immediate consequence.

\begin{lemma}
\label{LemPIE}
If $f_1, \ldots, f_n$ and $e$ are projections in a 
von-Neumann algebra $W$ where $f_i \precsim e$ for $1 \leq i \leq
n$ and $e$ is properly infinite 
then $\bigvee_{i=1}^n f_i \precsim e$.
\end{lemma}

For an element $x$ in a von-Neumann algebra $W$ the range projection
$r(x^*x)$ of $x^*x$ is the least projection $p$ in $W$ with $xp=x$,
and $r(x^*x) \sim r(xx^*)$
\cite[Proposition V.1.5]{TakesakiVolI}.

\begin{lemma}
\label{LemPIF}
Let $S$ be a finite subset of
$eW$ where 
$e$ is a properly infinite projection in the von-Neumann algebra $W$.
Then $S$ is contained in a universally reversible \JCstar-subtriple of
$eW$.
\end{lemma}

\begin{proof}
Let $x_1, \ldots, x_n$ be the elements of $S$.
For $1 \leq i \leq n$ we have $r(x_i^* x_i) \sim r (x_i x_i^*) \leq e$
so that by Lemmas~\ref{LemPID} and \ref{LemPIE}
there is a projection $f$ in
$W$ with $\bigvee_{i=1}^n r(x_i^* x_i) \leq f \sim e$. Hence for $1
\leq i \leq n$, $x_i \in eW \cap Wf = eWf$, whence the result by
Lemma~\ref{SquareRev}.
\end{proof}

\begin{proposition}
\label{PropPIG}
Let $e$ be a
projection in a von-Neumann algebra $W$
such that (a) $e$ is
properly infinite or (b) $W$ is continuous.
Then $eW$ is
a universally reversible \JCstar -triple.
\end{proposition}

\begin{proof}
If $e$ is properly infinite and $S$ is a finite subset of $\pi(eW)$,
where $\pi \colon eW \to \BH$ is a triple homomorphism, then
Lemma~\ref{LemPIF} implies the existence of a universally reversible
\JCstar-subtriple $M$ of $eW$ such that $S \subset \pi(M)$, implying
that the latter is reversible in $\BH$ and proving the result in case
(a).
The remaining case follows from this together with
Proposition~\ref{ContIsRev}.
\end{proof}

We now state and prove the main result of this section.

\begin{theorem}
\label{thm:URifNoTypIfin}
Every \JWstar-triple with zero type I$_{\mathrm{finite}}$ part is
a universally reversible \JCstar-triple.
\end{theorem}

\begin{proof}
The \JWstar-triples in question are those of the form
\[
(\mbox{type }I_\infty) \oplus_\infty (\mbox{continuous})
\]
(it being understood that one or both summands can be zero).
Recalling the
Horn-Neher structure theory related in
(\ref{eqn:typeIinfinity})
and
(\ref{eqn:continuousdecomp}),
the assertion follows from
Proposition~\ref{PropPIG} together with
Proposition~\ref{prop:revjcstaralg}.
\end{proof}

\begin{definition}
\label{def:wur}
We say that a \JWstar-triple $M$ is \emph{weakly universally reversible} if for
every weak*-continuous triple homomorphism $\pi \colon M \to \BH$, $\pi(M)$ is
reversible in $\BH$.
\end{definition}

\begin{remarks}
\label{rems:ur}
By \cite[Theorem 5.6]{BunceFeelyTimoney}
a Cartan factor is universally reversible if and only if it is not a Hilbert
space of dimension $\geq 3$ nor a spin factor of dimension $\geq 5$.
If $B$ is an abelian
von Neumann algebra and $C$ is a Cartan factor, then
\cite[Corollary 4.10]{BunceTimoney3} (for the case $\dim C < \infty$)
together with Proposition~\ref{PropPIG} (for the case $C = \BH e$)
and Proposition~\ref{prop:revjcstaralg} (for the remaining cases)
imply that $B \BARotimes C$ is 
universally reversible if and only if $C$ is universally reversible.
We note also that the class of \JWstar-triples satisfying
Definition~\ref{def:wur} is stable under arbitrary $\ell_\infty$ direct
sums.

Combining these remarks with
\cite[Theorem 4.12]{BunceTimoney3} and
Theorem~\ref{thm:URifNoTypIfin}, and using the notation $V_n$ for a spin factor
of dimension $n+1$, we have the following characterisation.
\end{remarks}

\begin{proposition}
\label{prop:wurcharact}
The following are equivalent for a \JWstar-triple $M$.
\begin{enumerate}[(a)]
\item
$M$ is  weakly universally reversible.
\item
The 
type I$_1$ and type
I$_{\mathrm{spin}}$ parts of $M$
are linearly isometric to
$
B_0 \oplus_\infty B_1 \otimes \ell_2^2$
and to
$ B_2 \otimes V_2
\oplus_\infty B_3 \otimes V_3$ (respectively),
where the $B_i$ are abelian von Neumann algebras (or zero).
\end{enumerate}

In particular $M$ is weakly universally reversible if
it has zero type I$_1$ and type
I$_{\mathrm{spin}}$ parts.
\end{proposition}

In the next section we shall exploit Proposition~\ref{prop:wurcharact}
to show amongst other things that weakly universally reversible \JWstar-triples
are universally reversible. For the time being we note the following.

\begin{proposition}
\label{URiffWURbidual}
Let $E$ be a \JCstar-triple. Then $E$  is universally reversible if and
only if its bidual $E^{**}$ is weakly universally reversible.
\end{proposition}

\begin{proof}
If $E$ is universally reversible and $\pi \colon E^{**} \to \BH$ is a
weak*-continuous triple homomorphism then, since $\pi(E)$ is reversible in
$\BH$, separate weak*-continuity of multiplication together with a simple
induction shows that $\pi(E^{**}) = \overline{\pi(E)}$ is reversible in $\BH$.

Conversely, let $E^{**}$ be weakly universally reversible.
Identifying $E$ with its image in $\TROu(E)$ (and $\TROu(E)$ canonically with
its embedding in its bidual, $E^{**}$ with the weak*-closure of $E$ in
$\TROu(E)^{**}$) we have the following commutative diagram
\[
\xymatrix{
E \ar@{^{(}->}[r] \ar@{^{(}->}[d] & \TROu(E) \ar@{^{(}->}[d]\\
E^{**} \ar@{^{(}->}[r] & \TROu(E)^{**}
}
\]
Let $x_1, \ldots, x_{2n+1} \in E$. Then, since $E^{**}$ is reversible in
$\TROu(E)^{**}$, by assumption,
$\trop{x_1}{\cdots}{x_{2n+1}}
+\trop{x_{2n+1}}{\cdots}{x_{1}} \in E^{**} \cap \TROu(E) = E$
(since elements of  $ \TROu(E) \setminus E$ are not in the
weak*-closure of $E$).
Therefore $E$ is reversible in $\TROu(E)$ and so is universally
reversible.
\end{proof}

\section{Characterising Universal Reversibility}
\label{sec:characterising}

We establish the prevalence of universally reversible \JCstar-triples
by proving that the property fails only when a \JCstar-triple
possesses a quotient (by an ideal) linearly isometric to a Hilbert
space of dimension $\geq 3$ or a spin factor of dimension $\geq 5$
(Theorem~\ref{thm:charaterise}).
In this sense we prove that `almost all' \JCstar-triples are
universally reversible.
In like vein we characterise universally reversible TROs and
\JWstar-triples, in the latter case showing that universal
reversibility is equivalent to the formally weaker version of
Definition~\ref{def:wur}.
We conclude this section
by showing that the number of terms in the definition of
universal reversibility can be reduced to the minimum possible of
five.

\begin{lemma}
\label{lem:spinTripotents}
Let $u$ be a nonzero tripotent in a \JCstar-subtriple $E$ of a spin
factor $V$.
Then
\begin{enumerate}[(a)]
\item
\label{lem:spinTripotentsi}
$u$ is unitary if it is not minimal;
\item
\label{lem:spinTripotentsii}
$E$ is linearly isometric to a spin factor, a Hilbert space or to $\IC
\oplus_\infty \IC$.
\end{enumerate}
\end{lemma}

\begin{proof}
\begin{enumerate}[(a)]
\item See \cite[Lemma 5.4]{EdwardsRutt1986}, for example.

\item
Since $V$ is reflexive, $E$ is a \JWstar-subtriple.
If $E$ is a factor and not a Hilbert space we may choose nonzero
orthogonal tripotents $u$ and $v$ in $E$ which, by 
(\ref{lem:spinTripotentsi}), are minimal in $V$ with $u+v$ unitary, so
that $E = E_2(u+v)$ is a spin factor.
If $E$ is not a factor, then $E = I \oplus_\infty J$ for certain
nonzero ideals $I$ and $J$ of $E$ and now we may choose minimal
tripotents $u \in I$ and $v \in J$ giving $E = E_2(u+v) = \IC u
\oplus_\infty \IC v$.
\end{enumerate}
\end{proof}

If $B$ is an abelian \Cstar-algebra and $C$ a finite dimensional
Cartan factor, we note that the \JCstar-triple $B \otimes C$ has a
separating family of Cartan factor representations onto $C$ given by
$\rho \otimes \id_C$ as $\rho$ ranges over the pure states of $B$.
Recall that we introduced following (\ref{eqn:TypeI}) the notion of a
\JWstar-triple being $\{C_j: j\in J\}$-homogeneous (for distinct
Cartan factors $C_j$).

\begin{lemma}
\label{subtripleOfFRRep}
Let $E$ be a \JCstar-triple with a separating family of Cartan factor
representations $\{ \pi_\lambda \colon E \to C_\lambda \}_{\lambda \in
\Lambda}$ where $C_\lambda$ has finite rank $\forall \lambda \in
\Lambda$. (Hence $\pi_\lambda(E) = C_\lambda$ for all $\lambda \in
\Lambda$.)
\begin{enumerate}[(a)]
\item 
\label{subtripleOfFRRepa}
If $D$ is a Cartan subfactor of $E$, then $D$ is linearly isometric
to a subfactor of $C_\lambda$ for some $\lambda \in \Lambda$.

\item
\label{subtripleOfFRRepb}
If $E$ is a \JWstar-triple and $F$ a \JWstar-subtriple of $E$, then
$F$ is $\{ D_i : i \in I \}$-homogeneous where each $D_i$ is contained
in $C_\lambda$ for some $\lambda$ (depending on $i$).

\item
\label{subtripleOfFRRepc}
If $\Lambda$ is finite and each $C_\lambda$ is finite-dimensional,
then
$E^{**}$ is
$\{ D_{1}, \ldots, D_{n} \}$-homogeneous, where each 
$D_i$ is contained
in $C_\lambda$ for some $\lambda$ (depending on $i$).
\end{enumerate}
\end{lemma}

\begin{proof}
\begin{enumerate}[(a)]
\item 
Let $D$ be a Cartan subfactor of $E$ and choose $\lambda \in \Lambda$
such that $\pi_\lambda$ does not vanish on the elementary ideal $K_D$
of $D$. Then (since $K_D$ is has no nontrivial ideals)
$K_D \cong \pi_\lambda(K_D) \subseteq C_\lambda$, implying
that $D$ has finite rank, giving $K_D = D$ and proving 
(\ref{subtripleOfFRRepa}).

\item
This follows from (\ref{subtripleOfFRRepa}) and Horn's type I
classification
(\ref{eqn:TypeI}) (since the existence of a weak*-continuous triple
homomorphism $\pi \colon F \to D$ onto a Cartan factor $D$ implies that
$D$ is isomorphic to a summand of $F$).

\item
In this case, for appropriate sets $X_\lambda$, $E$ may be
represented as a subtriple of a finite $\ell_\infty$-sum
$\bigoplus_{\lambda \in \Lambda} \ell_\infty (X_\lambda) \otimes
C_\lambda$ and , in turn, $E^{**}$ may be mapped to a
\JWstar-subtriple of $\bigoplus_{\lambda \in \Lambda} B_\lambda
\otimes C_\lambda$ (with $B_\lambda$ abelian von Neumann algebras).
The assertion now follows from (\ref{subtripleOfFRRepb}).
\end{enumerate}
\end{proof}

\begin{lemma}
\label{lem:TypeI1sub}
Let $M$ be a type I$_1$ \JWstar-triple.
Then every \JWstar-subtriple of $M$ is type I, $M^{**}$ is type I,
and all Cartan factor representations of $M$ are onto Hilbert spaces. 
\end{lemma}

\begin{proof}
The first assertion follows from
 Lemma~\ref{subtripleOfFRRep} (\ref{subtripleOfFRRepb}).
Since $M$ contains an abelian complete tripotent $u$, which is
automatically complete and abelian in $M^{**}$
(Lemma~\ref{lem:tripots}), the latter is type I.
Finally, for any Cartan factor representation $\pi \colon M \to C$,
$\pi(u)$ is a minimal complete tripotent in $C$, so that $C$ is a
Hilbert space.
\end{proof}

\begin{proposition}
\label{prop:typeI1}
The following are equivalent for a \JCstar-triple $E$.
\begin{enumerate}[(a)]
\item
All Cartan factor representations of $E$ are onto Hilbert
spaces.
\item
$E$ has a separating family of Hilbert space representations.
\item
$E^{**}$ is type I$_1$.
\end{enumerate}
\end{proposition}

\begin{proof}
(a) $\Rightarrow$ (b):
This is clear.

(b) $\Rightarrow$ (c):
Given (b), $E$ may be realised as a subtriple of an $\ell_\infty$-sum,
$M$ (say),
of Hilbert spaces.
Since $M$ is a type I$_1$ \JWstar-triple and $E^{**} \subset M^{**}$,
$E^{**}$ is type I$_1$ by 
Lemma~\ref{lem:TypeI1sub}.

(c) $\Rightarrow$ (a):
The implication follows from 
Lemma~\ref{lem:TypeI1sub} (since any Cartan factor representation $\pi
\colon E \to C$ extends to a triple homomorphism from $E^{**}$ onto
$C$).
\end{proof}

\begin{proposition}
\label{prop:NonTypeI1Ideal}
Let $E$ be a \JCstar-triple. Let $J$ be the intersection of the
kernels of all nonzero Hilbert space representations of $E$
(it being understood that $J = E$ if there are no nonzero Hilbert
space representations). Then
\begin{enumerate}[(a)]
\item
\label{prop:NonTypeI1Ideala}
$E^{**}$ has nonzero type I$_1$ part if and only if
$E$ has a nonzero Hilbert space representation.
\item
\label{prop:NonTypeI1Idealb}
$(E/J)^{**}$ is type I$_1$ and $J^{**}$ has zero type I$_1$ part.

\item
\label{prop:NonTypeI1Idealc}
$E^{**}$ has nonzero type I$_{\mathrm{spin}}$ part
if and only if $E$ has a spin factor representation.
\end{enumerate}
\end{proposition}

\begin{proof}
\begin{enumerate}[(a)]
\item
Suppose the type I$_1$ part, $M$, of $E^{**}$ is nonzero and consider
the (weak*-continuous) canonical projection $P \colon E^{**} \to M$.
Since $P(E)$ is a nonzero \JCstar-subtriple of $M$, by
Lemma~\ref{lem:TypeI1sub} there is a Hilbert space representation $\pi$
of $M$  which does not vanish on $P(E)$, implying that
$\pi \circ P$ is a Hilbert space representation of $E$.
The converse is clear.

\item
By construction $E/J$ has a separating family of Hilbert space
representations so that $(E/J)^{**} \equiv E^{**}/J^{**}$
is type I$_1$ by
Proposition~\ref{prop:typeI1}.
Further, since any Hilbert space representation of $J$ will extend to
one of $E$, $J^{**}$ has zero type I$_1$ part by
(\ref{prop:NonTypeI1Ideala}).

\item
The converse being clear, suppose that the type
I$_{\mathrm{spin}}$ part,
$N$,
of $E^{**}$ is nonzero.
By (\ref{prop:NonTypeI1Idealb}), 
$J^{**}$ is the orthogonal complement of the type I$_1$ part of
$E^{**}$ and so contains $N$ as a weak*-closed ideal.
Let $Q \colon E^{**} \to N$ be the canonical projection.
Since $N$ has a separating family of spin factor representations
\cite[Lemma 2]{Stacey1982},
we may choose a spin factor representation $\pi \colon N \to V$ with
$\pi(Q(J)) \neq \{0\}$. By the definition of $J$ together with
Lemma~\ref{lem:spinTripotents}, $\pi \circ Q$ induces a spin factor
representation of $J$, which extends to a spin factor representation
of $E$ (\cite[Remarks 2.5 (b)]{BunceTimoney3}).
\end{enumerate}
\end{proof}

\begin{lemma}
\label{lemma:suf4ur}
Let $E$ be a \JCstar-triple. Each of the following two conditions
separately implies that $E$ is universally reversible.
\begin{enumerate}[(a)]
\item
\label{lemma:suf4ura}
$E$ has separating family of representations onto Cartan factors $C
\in \{ \IC, \ell_2^2, V_2, V_3\}$
\item
\label{lemma:suf4urb}
$E$ has no nonzero Hilbert space representations and
no spin factor representations.
\end{enumerate}
\end{lemma}

\begin{proof}
By Proposition~\ref{URiffWURbidual},
in either case it is enough to show that $E^{**}$ is weakly
universally reversible.
\begin{enumerate}[(a)]
\item
In this case by Lemma~\ref{subtripleOfFRRep} (\ref{subtripleOfFRRepc})
and
Lemma~\ref{lem:spinTripotents}
$E^{**}$ is linearly isometric to an $\ell_\infty$-sum of
\JWstar-triples of the form $B \BARotimes C$ where $C \in \{ \IC,
\ell_2^2, V_2, V_3\}$ and $B$ is an abelian von Neumann algebra, and
so is universally reversible
(see Remarks~\ref{rems:ur}).

\item
By Proposition~\ref{prop:NonTypeI1Ideal}
(\ref{prop:NonTypeI1Ideala}),
(\ref{prop:NonTypeI1Idealc})
$E^{**}$ has zero type I$_1$ and type I$_{\mathrm{spin}}$
parts and thus is weakly
universally reversible by Proposition~\ref{prop:wurcharact}.
\end{enumerate}
\end{proof}

Our characterisation of universally reversible \JCstar-triples below
shows that Hilbert spaces of dimension $\geq 3$ and spin factors of
dimension $\geq 5$ are essentially the only obstacles to the property.

\begin{theorem}
\label{thm:charaterise}
Let $E$ be a \JCstar-triple. Then $E$ is universally reversible if and
only if $E$ satisfies both of the following conditions.
\begin{enumerate}[(a)]
\item
\label{thm:charaterisea}
$E$ has no representation onto a Hilbert space of dimension $\geq 3$.
\item
\label{thm:charateriseb}
$E$ has no representation onto a spin factor of dimension $\geq 5$.
\end{enumerate}
\end{theorem}

\begin{proof}
Let $E$ satisfy 
(\ref{thm:charaterisea}) and (\ref{thm:charateriseb}).
Let $J$ be as in
Proposition~\ref{prop:NonTypeI1Ideal}.
Note that the hypotheses pass to ideals (since Cartan
factor representations of ideals extend) and we rely on the fact that
universal reversibility of both an ideal $J$ of a \JCstar-triple $E$
and the quotient $E/J$ imply universal reversibility of $E$
\cite[Corollary 3.4]{BunceTimoney3}.

Now either $J = E$ (in which case $E/J$ is trivially reversible) or
$E/J$ has a separating family of representations onto Hilbert spaces
of dimensions one or two and thus is universally reversible by
Lemma~\ref{lemma:suf4ur} (\ref{lemma:suf4ura}). We show now
that $J$ is universally reversible.
Suppose that $J \neq \{0\}$ and let $K$ be the intersection of the
kernels of spin factor representations of $J$. Then $J/K$ is zero or
has a separating family of representations onto spin factors of
dimensions 3 or 4 so that $J/K$ is universally reversible, again by
Lemma~\ref{lemma:suf4ur} (\ref{lemma:suf4ura}).
The ideal $K$ is universally reversible because it satisfies the
condition of Lemma~\ref{lemma:suf4ur} (\ref{lemma:suf4urb}).
Hence $J$ is universally reversible, as therefore is $E$.

The converse is clear.
\end{proof}

\begin{remark}
\label{remark:simpleUR}
It is immediate from Theorem~\ref{thm:charaterise} that the only simple \JCstar-triples 
failing to be universally reversible are the Hilbert spaces of dimension
$\geq 3$ and the spin factors of dimension $\geq 5$.

Since realisable as a corner of a C*-algebra (see
\cite[p.~493]{EffrosOzawaRuan}, for example), up
to complete isometry every TRO is the range of a completely
contractive projection on a \Cstar-algebra. If $E$ is a \JCstar-triple then the
set of points in $\TROu(E)$ fixed by its canonical involution $\Phi$
is the range
of the bicontractive projection $\frac 12 ( \id  +  \Phi)$ on $\TROu(E)$.
It follows that
every universally reversible \JCstar-triple linearly isometric to the range of 
a contractive projection on a \Cstar-algebra. Since Hilbert spaces and spin
factors also possess this property, by
\cite{EffrosStormerMathScand1979}
in the spin case, we note that every simple \JCstar-triple
has the same property.
\end{remark}

We recall from \cite[Theorem 3.5]{BunceTimoney3} that every
\JCstar-triple has a largest universally reversible
ideal. We can reprove its existence and characterise it as
follows.

\begin{corollary}
Let $E$ be a \JCstar-triple and let $I$ denote the intersections of
the kernels of
representations onto Hilbert space of dimension $\geq 3$ and spin factors of dimension
$\geq 5$
(understood as $E$ if there are no such
representations). Then $I$ is the largest universally reversible
ideal of $E$.
\end{corollary}

\begin{proof}
If $\pi \colon E \to C$ is a Cartan factor representation such that
$C$ fails to be universally reversible,
then $C$ has finite rank and so $\pi(E) = C$. Hence if $J$ is an
ideal in $E$ we have $\pi(J)$ an ideal in $C$, giving $\pi(J)$ equal
to $\{0\}$ or $C$. If $J$ is universally reversible we must have $J
\subseteq \ker \pi$ for all such $\pi$ and so $J \subseteq I$.

On the other hand,
by the extension property of
Cartan factor representations from ideals
and the construction of $I$,
Theorem
\ref{thm:charaterise} implies that
$I$ is universally
reversible.
\end{proof}

\begin{corollary}
\label{cor:M2uno1D}
Let $E$ be a \JCstar-triple with a complete tripotent $u$. If $E_2(u)$
is universally reversible with no one dimensional representations,
then $E$ is universally reversible.
\end{corollary}

\begin{proof}
Let $\pi \colon E \to C$ be a Hilbert space or spin factor
representation and put $v
= \pi(u)$.
Then $v$ is a complete tripotent of $C$ and the induced map $\pi_2
\colon M_2(u) \to C_2(v)$ is surjective (since $\pi(E) = C$).
If $C$ is a spin factor, then $C = C_2(v)$ is universally reversible
and so has dimension 3 or 4. If $C$ is a (nonzero) Hilbert space
then $C_2(v) =\IC v$, a contradiction. Hence, $E$ is universally
reversible by Theorem~\ref{thm:charaterise}.
\end{proof}

By passing to a quotient, if a spin factor is a triple homomorphic
image of a TRO, then it is linearly isometric to a TRO and thus, via
Proposition~\ref{prop:TROimage},
Jordan *-isomorphic to a type I$_2$ von Neumann algebra, and hence to
$M_2(\IC)$.

\begin{theorem}
\label{thm:TROurcharact}
Let $T$ be a TRO. Then the following are equivalent.
\begin{enumerate}[(a)]
\item
\label{thm:TROurcharacta}
$T$ is universally reversible (as a \JCstar-triple).
\item
\label{thm:TROurcharactb}
$T$ has no triple homomorphisms onto a Hilbert space of dimension
$\geq 3$.
\item
\label{thm:TROurcharactc}
$T$ has no TRO homomorphisms onto a Hilbert space of dimension
$\geq 3$.
\end{enumerate}
\end{theorem}

\begin{proof}
The equivalence of 
(\ref{thm:TROurcharacta})
and
(\ref{thm:TROurcharactb})
follows from Theorem~\ref{thm:charaterise}
and the above remark.
That
(\ref{thm:TROurcharactc})
$\Rightarrow$
(\ref{thm:TROurcharactb})
is seen on passing to quotients (recalling that triple ideals in a TRO
are TRO ideals (\cite[Proposition 5.8]{Harris81}), and
(\ref{thm:TROurcharactb})
$\Rightarrow$
(\ref{thm:TROurcharactc})
is clear.
\end{proof}

\begin{lemma}
\label{lem:typeI2ur}
Let $M$ be an $\ell_\infty$-sum $\bigoplus_{i \in I} B_i \BARotimes
\mathcal{B}(H_i)e_i$
where $B_i$ is an abelian von Neumann algebra and $e_i$ is a rank 2
projection in $\mathcal{B}(H_i)$, for each $i \in I$.
Then $M$ is universally reversible.
\end{lemma}

\begin{proof}
With $W = \bigoplus_{i \in I} B_i \BARotimes
\mathcal{B}(H_i)$ and $e = \bigoplus_{i
\in I} 1 \otimes e_i$, we have $e$ a complete tripotent in $M = We$ and $M_2(e) =
eWe$ is a type I$_2$ von Neumann algebra.
Thus $M_2(e)$ is universally reversible and
has no
*-homomorphism onto $\IC$.
Hence $M$ is universally reversible by Corollary~\ref{cor:M2uno1D}.
\end{proof}

If $e$ and $f$ are projections in a \JWstar-algebra $M$ we use $e
\simone f$ to denote that $e$ and $f$ are exchanged by a symmetry
and we recall \cite[Lemma 5.2.9]{HOS}
that if $(e_i)_{i \in I}$ and
$(f_i)_{i \in I}$ are two families of orthogonal projections in $M$
with $e_i \simone f_i$ for all $i \in I$ and $\sum_{i \in I} e_i \perp
\sum_{i \in I} f_i$, then
$\sum_{i \in I} e_i \simone \sum_{i \in I} f_i$.

\begin{lemma}
\label{lem:jwstarnoI1noSp}
Let $M$ be a type I finite \JWstar-algebra with no nonzero summands of
type I$_1$ nor of type I$_{\mathrm{spin}}$. Then $M$ is a universally
reversible \JCstar-triple with no one dimensional representations.
\end{lemma}

\begin{proof}
We may suppose that there is a sequence $(z_i)_i$ of orthogonal
central projections in $M$ with (weak) sum 1 such that each $Mz_i$ is
a type I$_{n_i}$ \JWstar-algebra with $n_i$ strictly increasing and $3
\leq n_i < \infty$, for all $i$.

Fixing $i$, we have
\begin{equation}
\label{eq:zi}
z_i = e_{i,1} + \cdots e_{i, n_i}
\end{equation}
where the $e_{i,j}$ ($1 \leq j \leq n_i$) are orthogonal and exchanged
by symmetries. Taking equivalence classes modulo 3, write $n_i = 3k+r$
where $r =0, 1 \mbox{ or } 3$.
Letting $f_{i, j} = \sum_{s=1}^{k} e_{i,3(s-1)+j}$ ($1 \leq j \leq 3$)
we then have
\[
z_i = f_{i,1} + f_{i,2}+ f_{i,3} + g_i
\]
where
$f_{i,1}$,
$f_{i,2}$ and
$f_{i,3}$ are exchanged by symmetries and
$g_i \simone p_i \leq f_{i,1} + f_{i,2}$
for some projection $p_i$ in $M$.
Repeating this process for each $i$ and putting
$f_j = \sum_i f_{i,j}$ for $1 \leq j \leq 3$, $g = \sum_i g_i$ and $p
= \sum_i p_i$ we have
\begin{equation}
\label{eq:fgp}
f_1+f_2 + f_3 + g =1, 
\,
f_j \simone f_m
\, (1 \leq j, m \leq 3)
\mbox{ and }
g \sim p \leq f_1 + f_2.
\end{equation}
The condition (\ref{eq:fgp}) continues to hold in $M^{**}$, showing
that $M^{**}$ has zero type I$_{\mathrm{spin}}$ part since,
for any nonzero central projection $z$ in $M^{**}$, $zf_1$, $zf_2$ and
$zf_3$ are nonzero orthogonal projections.
Therefore, by Proposition~\ref{prop:wurcharact},
$M^{**}$ is weakly universally reversible.
Hence $M$ is a universally reversible \JCstar-triple
by Proposition~\ref{URiffWURbidual}.

By a similar argument, any quotient of $M$ by an ideal must have
dimension at least 3.
\end{proof}

\begin{theorem}
\label{thm:jwstarurchar}
The following are equivalent for a \JWstar-triple $M$.
\begin{enumerate}[(a)]
\item
\label{thm:jwstarurchara}
$M$ is universally reversible.
\item
\label{thm:jwstarurcharb}
$M$ is weakly universally reversible.
\item
\label{thm:jwstarurcharc}
The type I$_1 \oplus_\infty \mathrm{I}_{\mathrm{spin}}$ part of $M$
is linearly isometric to
\[
B_0 \oplus_\infty B_1 \otimes \ell_2^2
\oplus_\infty B_2 \otimes V_2
\oplus_\infty B_3 \otimes V_3,
\]
where the $B_i$ are abelian von Neumann algebras (or zero).
\end{enumerate}
\end{theorem}

\begin{proof}
(\ref{thm:jwstarurcharc}) $\Rightarrow$ (\ref{thm:jwstarurchara}):
Assume (\ref{thm:jwstarurcharc}).
Let $J$
denote the
type I$_1 \oplus_\infty \mathrm{I}_{\mathrm{spin}}$ part of $M$.
By 
Remarks~\ref{rems:ur}, $J$ is universally reversible and, passing to
the orthogonal complement of $J$ in $M$, we may suppose $J = \{0\}$.
By Theorem~\ref{thm:URifNoTypIfin}, we may assume further that
$M = \bigoplus_i B_i \BARotimes C_i$,
where the $B_i$ are abelian von Neumann algebras and the $C_i$ are
Cartan factors with $3 \leq \rank (C_i) < \infty$ for all $i$.
For each $i$, choose a complete tripotent $u_i \in B_i \BARotimes C_i$
and put $u = \bigoplus_i u_i$.
Then $u$ is a complete tripotent of $M$ with $M_2(u)$ a
\JWstar-algebra of the form considered in
Lemma~\ref{lem:jwstarnoI1noSp}, and hence $M_2(u)$ is universally
reversible with no one dimensional representations.
By Corollary~\ref{cor:M2uno1D}, $M$ is universally reversible.

The implication
(\ref{thm:jwstarurchara}) $\Rightarrow$ (\ref{thm:jwstarurcharb}) is
clear, and
(\ref{thm:jwstarurcharb}) $\Rightarrow$ (\ref{thm:jwstarurcharc})
is given by Proposition~\ref{prop:wurcharact}.
\end{proof}

\begin{corollary}
\label{corol:wstartroUR}
A \Wstar-TRO is universally reversible (as a \JCstar-triple) if and
only if its type I$_1$ part is linearly isometric to $B_0
\oplus_\infty B_1 \BARotimes \ell_2^2$, where $B_0$ and $B_1$ are
abelian von Neumann algebras.
\end{corollary}

Given a \JCstar-triple $E$ and a Cartan factor $C$, we have that $C$
is linearly isometric to a weak*-closed ideal of $E^{**}$ if and only
if there is a Cartan factor representation $\pi \colon E \to C$.
Recall (see \cite[Proposition 2 and Corollary 4]{FRgn}) that the
weak*-closed linear space of the minimal tripotents in $E^{**}$ is
called the \emph{atomic part $E^{**}_{\mathrm{at}}$}
of $E^{**}$ and $E^{**}_{\mathrm{at}}$ is the $\ell_\infty$ sum of the
distinct Cartan factor ideals of $E^{**}$.
Combining this with
Theorems~\ref{thm:charaterise} and \ref{thm:jwstarurchar}
we have the following.

\begin{corollary}
The following are equivalent for a \JCstar-triple $E$.
\begin{enumerate}[(a)]
\item $E$ is universally reversible.
\item $E^{**}_{\mathrm{at}}$ is universally reversible.
\item $E^{**}$ is universally reversible.
\end{enumerate}
\end{corollary}

We shall conclude this section by showing that the number of 
terms in the definition of universally reversible \JCstar-triples can
be reduced to the smallest possible.
We recall that the
universal TRO
$(\TROu(A), \utro A)$ coincides with the
universal \Cstar-algebra 
$ (C^*_{\mathrm{J}}(A), \beta_A)$
when $A$ is a \JCstar-algebra \cite[Proposition 3.7]{BunceFeelyTimoney}.

\begin{theorem}
\label{thm:5term}
The following are equivalent for a \JCstar-triple $E$.
\begin{enumerate}[(a)]
\item
\label{thm:5terma}
$E$ is universally reversible.
\item
\label{thm:5termb}
$\trop{x_1}{\cdots}{x_5} +
\trop{x_5}{\cdots}{x_1}  \in \utro E(E)$ whenever
$x_1, \ldots, x_5 \in \utro E(E)$.

\item
\label{thm:5termc}
$\trop{y_1}{\cdots}{y_5} +
\trop{y_5}{\cdots}{y_1}  \in \pi (E)$ whenever
$y_1, \ldots, y_5 \in \pi(E)$ and $\pi \colon E \to \BH$ is a triple
homomorphism.
\end{enumerate}
\end{theorem}

\begin{proof}
(\ref{thm:5terma}) $\Rightarrow$ (\ref{thm:5termb})
and
(\ref{thm:5termc}) $\Rightarrow$ (\ref{thm:5termb})
are clear.

(\ref{thm:5termb}) $\Rightarrow$ (\ref{thm:5termc}):
Assume (\ref{thm:5termc}).
Let $\pi \colon E  \to \BH$ be a triple homomorphism, $y_i \in \pi(E)$
and $x_i \in E$ with $\pi(x_i) = y_i$ for $1 \leq i \leq 5$.
Let $\tilde{\pi} \colon \TROu(E) \to \BH$ be the TRO homomorphism with
$\pi = \tilde{\pi} \circ \utro E$. Then with $z_i = \utro E(x_i)$,
applying $\tilde{\pi}$ to
$\trop{z_1}{\cdots}{z_5} +
\trop{z_5}{\cdots}{z_1}  \in \utro E(E)$ gives
$\trop{y_1}{\cdots}{y_5} +
\trop{y_5}{\cdots}{y_1}  \in \tilde{\pi}(\utro E (E)) = \pi (E)$.

(\ref{thm:5termc}) $\Rightarrow$ (\ref{thm:5terma}):
Assume (\ref{thm:5termc}).
Let $\psi \colon E \to F \subset \BH$ be a surjective triple
homomorphism and consider
$\pi \colon E \to \utro F(F) \subset \TROu(F)$ where $\pi$ denotes
$\utro F \circ \psi$.
By 
\cite[Theorem 5.1 and proof]{BunceFeelyTimoney},
if $F$ is a Hilbert space of dimension $\geq 3$ then $\utro F \colon F
\to \TROu(F)$ fails condition (\ref{thm:5termc}), as therefore does
$\pi$.
If $F$ is a spin factor of dimension $\geq 5$ then (using
$(\TROu(F), \utro F)
= (C^*_{\mathrm{J}}(F), \beta_F)$)
$\utro F(F_{sa}) = \utro F (F)_{sa}$ is not reversible in $\TROu(F)$
so that by \cite{Cohn1954}
(or see \cite[p. 142]{HOS})
there exist $y_1, y_2, y_3, y_4 \in 
\utro F (F)_{sa}$ with
$y_1y_2y_3y_4
+y_4y_3y_2y_1 \notin \utro F (F)_{sa}$.
Thus, in this case $\pi$ fails condition
(\ref{thm:5termc}) for $y_1, y_2, y_3, y_4, y_5$ with $y_5 = 1$.

Therefore $E$ is universally reversible by
Theorem~\ref{thm:charaterise}.
\end{proof}

\section{\JC-operator spaces}
\label{sec:JCOpspaces}

We recall from \S\ref{sec:preliminaries} that the \JC-operator space
structures of a \JCstar-triple $E$, the operator space structures
determined by triple embeddings $\pi \colon E \to \BH$, are the
$E_{\mathcal{I}}$ induced by the maps $E \to \TROu(E)/\mathcal{I}$
($x \mapsto \utro E(x) + \mathcal{I}$)
as
$\mathcal{I}$ ranges over the ideals of $\TROu(E)$ having vanishing
intersection with $\utro E(E)$ (the operator space ideals of
$\TROu(E)$).

Let $T$ be a subTRO of $\BH$. We use $T^\OP$ to denote the identical
image of $T$ in the opposite \Cstar-algebra $\BH^\OP$. Given a
transposition $x \mapsto x^t$ on $\BH$ we write $T^t = \{ x^t : x \in
\BH\}$. Thus $\id \colon T \to T^\OP$ and the map
$T^\OP \to T^t$ ($x \mapsto x^t$) are a TRO anti-isomorphism and a TRO
isomorphism, respectively.

By definition, $T$ is an \emph{abelian TRO} if $\trop x y z = \trop z y
x$ for all $x, y, z \in T$.
Evidently, $T$ is an abelian TRO if and only if $\id \colon T \to
T^\OP$ is a TRO isomorphism (equivalently, a complete isometry).

\begin{lemma}
\label{lem:TTop}
Let $T$ be a TRO. Then
\begin{enumerate}[(a)]
\item 
\label{lem:TTopa}
$\id \colon T \to T^\OP$ is completely contractive if and only if $T$
is abelian;

\item
\label{lem:TTopb}
a TRO antihomomorphism $\pi \colon T \to \BH$ is completely
contractive if and only if $\pi(T)$ is abelian.
\end{enumerate}
\end{lemma}

\begin{proof}
\begin{enumerate}[(a)]
\item 
If $\id \colon T \to T^\OP$ is completely contractive, then so is 
$\id \colon T^\OP \to (T^\OP)^\OP = T$ by 
\cite[1.1.25]{BlecherleMerdy},
implying that it is a complete isometry, and hence that $T$ is
abelian. The converse is clear.

\item
Let $\pi \colon T \to \BH$ be a completely contractive TRO
antihomomorphism. Let $S$ denote $T/\ker \pi$ and let $\tilde\pi
\colon S \to \pi(T)$ denote the induced TRO antihomomorphism, also
completely contractive (by
\cite[1.1.15]{BlecherleMerdy}).
Since $\tilde \pi^{-1} \colon \pi(T) \to S^\OP$ is a TRO isomorphism,
$\tilde{\pi}^{-1} \circ \pi \colon S \to S^\OP$, which is the identity
map, is completely contractive, so that $S$ is abelian by
(\ref{lem:TTopa}).
Conversely, a TRO antihomomorphism into an abelian TRO is a TRO
homomorphism and so is completely contractive.
\end{enumerate}
\end{proof}

A \JCstar-triple $E$ is said to be \emph{abelian} if
\[
\tp x y {\tp a b c} = \tp {\tp x y a} bc
\]
for all $x,y,a,b,c \in E$.
An abelian TRO is an abelian \JCstar-triple. On the other hand, given
an abelian \JCstar-subtriple $E$ of $\BH$, the abelian \JWstar-triple
$\bar{E}$
(the weak*-closure of $E$)
is linearly isometric to an abelian von Neumann algebra by
\cite[(3.11)]{HornPredual}
and we may choose a unitary tripotent $u \in \bar{E}$.
Then $\bar{E}_2(u)$  is an abelian von Neumann subalgebra of
$\BH_2(u)$, the binary product and involution of the latter being
given by
$x \bullet y = \trop x u y$ and $x^\# = \trop u x u$.
In particular, with $a, b, c \in E$ (since $E \subseteq \bar{E} =
\bar{E}_2(u)$) we have
$\trop a b c = ab^*c = a \bullet b^\# \bullet c = c \bullet b^\#
\bullet a = cb^*a$. Thus $E$ is a TRO and the inclusion $E
\hookrightarrow \bar{E}_2(u)$ is a TRO isomorphism onto a subTRO of
$\bar{E}_2(u)$.

\begin{proposition}
\label{prop:abeliancase}
The following are equivalent for a \JCstar-subtriple $E$ of $\BH$.
\begin{enumerate}[(a)]
\item
\label{prop:abeliancasea}
$E$ is an abelian \JCstar-triple.
\item
\label{prop:abeliancaseb}
$E$ is an abelian TRO.
\item
\label{prop:abeliancasec}
$E$ is completely isometric to a subTRO of an abelian \Cstar-algebra.
\item
\label{prop:abeliancased}
$E$ has a separating family of representations onto $\IC$.
\item
\label{prop:abeliancasee}
$\TROu(E) = E$
\end{enumerate}
\end{proposition}

\begin{proof}
The equivalence of
(\ref{prop:abeliancasea}),
(\ref{prop:abeliancaseb}) and
(\ref{prop:abeliancasec}) is verified by the preamble, and
(\ref{prop:abeliancasea}) $\iff$
(\ref{prop:abeliancased})
is shown in
\cite[Proposition 6.2]{KaupFibreBund}.
If 
(\ref{prop:abeliancaseb}) holds
then every triple homomorphism from $E$ into a \Cstar-algebra is a TRO
homomorphism, giving
(\ref{prop:abeliancasee}).
Conversely
(\ref{prop:abeliancasee})
implies that $E$ is a TRO and the universal property
of $\TROu(E)$ implies that $\id \colon E \to E^\OP$ is a TRO
isomorphism, giving 
(\ref{prop:abeliancaseb}).
\end{proof}

\begin{corollary}
Every abelian \JCstar-triple has a unique \JC-operator space
structure.
\end{corollary}

\begin{theorem}
\label{thm:TROuofT}
Let $T \subset \BH$ be a TRO with no nonzero representations onto a Hilbert space
of any dimension other (possibly)
than two. Suppose $x \mapsto x^t$ is a
transposition of $\BH$. Then
\begin{enumerate}[(a)]
\item
\label{thm:TROuofTa}
$\TROu(T) = T \oplus T^t$ with $\utro T (x) = x \oplus x^t$ and
canonical involution $\Phi(x \oplus y^t) = y \oplus x^t$ (for $x, y
\in T$);
\item
\label{thm:TROuofTb}
$T$ has at least three distinct \JC-operator space structures.
\end{enumerate}
\end{theorem}

\begin{proof}
\begin{enumerate}[(a)]
\item
By the assumptions on $T$ and Theorem~\ref{thm:TROurcharact}, $T$ is
universally reversible with no ideals of codimension one. Moreover $x
\mapsto x^t$ is a TRO antiautomorphism of $\BH$.
Therefore
(\ref{thm:TROuofTa}) follows from \cite[Corollary
4.5]{BunceFeelyTimoney}.

\item
Regarding $T$ and $T^t$ as operator subspaces of $\BH$ and $T \oplus
T^t$ as an operator subspace of $\BH \oplus \BH$, neither the map $T
\to T^t$ ($x \mapsto x^t$) nor its inverse can be completely
contractive, by 
Lemma~\ref{lem:TTop} (\ref{lem:TTopa}), since $T$ is not abelian.
Further, since the natural TRO projection $T \oplus T^t \to T^t$ is
completely contractive, $\utro T \colon T \to T \oplus T^t$
cannot be completely contractive. Thus if
$\mathcal{I}_0$ denotes the zero ideal and $\mathcal{I}_1 = T \oplus
\{0\}$, $\mathcal{I}_2 = \{0\} \oplus T^t$ (which are nonzero operator
space ideals of $\TROu(T)$), then the \JC-operator spaces
$T_{\mathcal{I}_j}$ ($j =0, 1, 2$) are mutually distinct.
\end{enumerate}
\end{proof}

Given an ideal $I$ of a \JCstar-triple $E$ we recall
\cite[Theorem 3.3]{BunceTimoney3}
that
the canonical involution $\Phi$ of $\TROu(E)$ restricts to
that of $\TROu(I)$ and
the canonical involution of $\TROu(E/I)= \TROu(E)/\TROu(I)$ is given by
$x + \TROu(I) \mapsto \Phi(x) + \TROu(I)$ ($x \in \TROu(E)$).
In each of  the following three technical results $\Phi$ denotes the
canonical
involution of $\TROu(E)$ for the \JCstar-triple $E$ in question.

\begin{lemma}
\label{lem:idsTROuE}
Let $E$ be a universally reversible \JCstar-triple and
let $\mathcal{I}$
be an ideal of $\TROu(E)$.
Then
\begin{enumerate}[(a)]
\item
\label{lem:idsTROuEa}
if $\mathcal{I}$ is an operator space ideal with $\Phi(\mathcal{I}) =
\mathcal{I}$ then $\mathcal{I} = \{0\}$;
\item
\label{lem:idsTROuEb}
$\mathcal{I} \cap \Phi( \mathcal{I} ) = \TROu(J)$ where $J$ is the
ideal of $E$ given by $\utro E (J) = \utro E(E) \cap \mathcal{I}$;

\item
\label{lem:idsTROuEc}
$( \mathcal{I} + \Phi(\mathcal{I})) \cap \utro E (E) = \{ x + \Phi(x):
x \in \mathcal{I} \}$.

\end{enumerate}
\end{lemma}

\begin{proof}
\begin{enumerate}[(a)]
\item
See
\cite[Lemma 4.3]{BunceFeelyTimoney}.

\item
Letting $\tilde{\Phi}$ be the canonical involution of 
$\TROu(E/J) = \TROu(E)/\TROu(J)$ and $\mathcal{K} = \mathcal{I} \cap \Phi(
\mathcal{I} )$,
$\mathcal{K}/ \TROu(J)$
is a $\tilde{\Phi}$-invariant operator space ideal
of $\TROu(E/J)$. Since $E/J$ is universally reversible, $\mathcal{K} = \TROu(J)$
by 
(\ref{lem:idsTROuEa}).

\item
Given $x, y \in \mathcal{I}$ with $x + \Phi(y) \in \utro E (E)$, we
have $x + \Phi(y) = a + \Phi(a)$ where $a =(1/2)(x+y)$, so that the
left hand side is contained in the right.
On the other hand, since $E$ is universally reversible, the right hand
side is contained in $\utro E(E)$, establishing the assertion.

\end{enumerate}
\end{proof}

\begin{proposition}
\label{prop:idsTROuE}
Let $E$ be a universally reversible \JCstar-triple.
Then the following are equivalent for an ideal
$\mathcal{I}$ of $\TROu(E)$
\begin{enumerate}[(a)]
\item
\label{prop:idsTROuEa}
$\mathcal{I}$ is an operator space ideal of $\TROu(E)$;
\item
\label{prop:idsTROuEb}
$\mathcal{I} \cap \Phi( \mathcal{I} ) = \{0\}$;
\item
\label{prop:idsTROuEc}
the map
$\mathcal{I} \to ( \mathcal{I} + \Phi(\mathcal{I}) ) \cap \utro E (E)$
($x \mapsto x + \Phi(x)$) is a triple isomorphism (onto an ideal of
$\utro E(E)$).
\end{enumerate}
\end{proposition}

\begin{proof}
(\ref{prop:idsTROuEa})
$\Rightarrow$
(\ref{prop:idsTROuEb}) is immediate from 
Lemma~\ref{lem:idsTROuE} (\ref{lem:idsTROuEa})
since $\mathcal{I} \cap \Phi( \mathcal{I} )$ is a $\Phi$-invariant
ideal of $\TROu(E)$.

(\ref{prop:idsTROuEb})
$\Rightarrow$
(\ref{prop:idsTROuEc}).
Given (\ref{prop:idsTROuEb})
we have $\mathcal{I} + \Phi(\mathcal{I}) =
\mathcal{I} \oplus_\infty \Phi(\mathcal{I}) $
so that the stated map is a linear isometry, surjective by
Lemma~\ref{lem:idsTROuE} (\ref{lem:idsTROuEc}), and thus a triple
isomorphism.

(\ref{prop:idsTROuEc})
$\Rightarrow$
(\ref{prop:idsTROuEa}).
If (\ref{prop:idsTROuEc}) holds and $x \in \mathcal{I} \cap \utro E
(E)$, then $\Phi(x) = x$ and $\|x\| = \|2x\|$, giving $x = 0$.
\end{proof}

\begin{lemma}
\label{lem:20120657}
Let $\mathcal{I}$, $\mathcal{J}$ be operator space ideals of
$\TROu(E)$ with $\mathcal{J} \subseteq \mathcal{I}$,
where
$E$ is a universally reversible \JCstar-triple.
Put $\mathcal{K} = \mathcal{J} + \Phi(\mathcal{J})$ and let $L$ be the
ideal of $E$ such that $\utro E (L) = \utro E(E) \cap \mathcal{K}$.
Then
\begin{enumerate}[(a)]
\item
\label{lem:20120657a}
$\mathcal{I} \cap (\utro E(E) + \mathcal{K}) = \mathcal{J} 
= \mathcal{I} \cap \mathcal{K}$;

\item 
\label{lem:20120657b}
$( \mathcal{I} + \mathcal{K} ) \cap (\utro E(E) + \mathcal{K})
= \mathcal{K}$;

\item
\label{lem:20120657c}
$\mathcal{I}/ \mathcal{J}$ is TRO isomorphic 
to an operator space ideal of $\TROu(E/L) = \TROu(E)/\mathcal{K}$.

\end{enumerate}
\end{lemma}

\begin{proof}
\begin{enumerate}[(a)]
\item
That $\mathcal{J}$ is contained in the two stated sets is clear.
Conversely given $x \in \mathcal{I} \cap ( \utro E(E) + \mathcal{K})$
we have $x = y + a + \Phi(b)$ for some $y$ ($=\Phi(y)$) in $\utro
E(E)$ and $a, b \in \mathcal{J}$. Putting $c = a -b $ we have
$\Phi(x -c) = x - c$ and so, since $E$ is universally reversible,
$x - c \in \utro E(E) \cap \mathcal{I} = \{ 0\}$ giving $x = c \in
\mathcal{I}$.

\item
Given $z$ in the left hand set we have $x \in \mathcal{I}$,
$y \in \utro e(E)$ and $a,b \in \mathcal{K}$ such that $z=
x +a = y + b$. Then $x = y + b - a \in \mathcal{I} \cap (\utro e(E) +
\mathcal{K}) =
\mathcal{J}$, by (\ref{lem:20120657a}),
so that $z \in \mathcal{K}$.

\item
The quotients 
$\mathcal{I}/\mathcal{J}$
and
$(\mathcal{I} + \mathcal{K})/\mathcal{K}$
are TRO isomorphic by 
(\ref{lem:20120657a}).
By Lemma~\ref{lem:idsTROuE}
(\ref{lem:idsTROuEb}),
$\TROu(L) = \mathcal{K}$, so that
$\TROu(E/L) = \TROu(E)/\mathcal{K}$.
Since $(\mathcal{I} + \mathcal{K})/\mathcal{K}$ has vanishing
intersection with
$(\utro E(E) + \mathcal{K})/\mathcal{K}$,
by (\ref{lem:20120657b}),
it is an operator space ideal of $\TROu(E/L)$.
\end{enumerate}
\end{proof}

\begin{lemma}
\label{lem:20120658}
Let $E$ be a \JCstar-triple and let $E_1$ denote the intersection of
the kernels of all the one dimensional representations of $E$ (with
$E_1 = E$ if $E$ has no representations onto $\IC$).
Then
\begin{enumerate}[(a)]
\item
\label{lem:20120658a}
$E_1$ has no (nonzero) one dimensional representations and $E/E_1$ is
abelian;

\item
\label{lem:20120658b}
$\TROu(E_1)$ has no  one dimensional representations
and $\TROu(E/E_1) = \TROu(E)/\TROu(E_1)$ is abelian;

\item
\label{lem:20120658c}
an ideal $\mathcal{I}$ of $\TROu(E)$ is abelian if and only if
$\mathcal{I} \cap \TROu(E_1) = \{0\}$.

\end{enumerate}
\end{lemma}

\begin{proof}
(\ref{lem:20120658a})
and
(\ref{lem:20120658b})
are immediate by Proposition~\ref{prop:abeliancase}
and the properties of the universal TRO.
To see
(\ref{lem:20120658c}),
if $\mathcal{I}$ has trivial intersection with $\TROu(E_1)$ then it
may be realised as  a subtriple of $\TROu(E)/\TROu(E_1)$ and so is
abelian by the second part of 
(\ref{lem:20120658b}). On the other hand the first part of
(\ref{lem:20120658b}) implies that $\TROu(E_1)$ has no nonzero abelian
ideal.
\end{proof}

\begin{proposition}
\label{prop:20120659}
Let $E$ be a universally reversible \JCstar-triple and let $E_1$ be as
in Lemma~\ref{lem:20120658}.
Then
\begin{enumerate}[(a)]
\item
\label{prop:20120659a}
$\TROu(E)$ has no nonzero abelian operator space ideals;

\item
\label{prop:20120659b}
all operator space ideals of $\TROu(E)$ are contained in $\TROu(E_1)$.

\end{enumerate}
\end{proposition}

\begin{proof}
\begin{enumerate}[(a)]
\item
Let $\mathcal{I}$ be an abelian operator space ideal of $\TROu(E)$ and
let $J$ be the ideal of $E$ with $\utro E(J) = \utro E(E) \cap
(\mathcal{I} + \Phi(\mathcal{I}))$.
Since $\mathcal{I}$ and $\Phi(\mathcal{I})$ are orthogonal,
by Proposition~\ref{prop:idsTROuE}
(\ref{prop:idsTROuEb}),
$\mathcal{I} + \Phi(\mathcal{I})$ is abelian, as therefore is $J$.
Hence, using
Proposition~\ref{prop:abeliancase}
(\ref{prop:abeliancasee}) in the first equality and 
Lemma~\ref{lem:idsTROuE}
(\ref{lem:idsTROuEb}) in the second,
we have
$\utro E(J) = \TROu(J) = \mathcal{I} + \Phi(\mathcal{I})$, so that
$\mathcal{I}$ is contained in  $\utro E(E)$. Hence $\mathcal{I} =
\{0\}$.

\item
Let $\mathcal{I}$
be a nonzero operator space ideal of $\TROu(E)$. Put
$\mathcal{J}  =   \mathcal{I} \cap ∩ \TROu(E_1)$,
$\mathcal{K}  =  \mathcal{J} + φ\Phi(\mathcal{J})$. 
[Note that $\mathcal{J}$ is nonzero by
(\ref{prop:20120659a})
together
with Lemma~\ref{lem:20120658}  (\ref{lem:20120658c}).]
By construction $\mathcal{I}/\mathcal{J}$
is TRO isomorphic to
$(\TROu( E_1) + \mathcal{J})/\TROu(E_1)$
and so is abelian by
Lemma~\ref{lem:20120658}  (\ref{lem:20120658b}).
In addition, by
Lemma~\ref{lem:20120657} (\ref{lem:20120657c}),
there is a \JCstar-triple quotient $F$ of $E$ such that
$\mathcal{I}/\mathcal{J}$
is TRO isomorphic to an operator space ideal of $\TROu(F)$
and therefore
vanishes, by (a), since $F$ is universally reversible.
Hence,   $\mathcal{I}  =  \mathcal{J}$.
\end{enumerate}
\end{proof}

Given a \JCstar-triple $E$ and an operator space ideal $\mathcal{I}$
of $\TROu(E)$, in the proof below
we denote the triple homomorphic image
$\{ \utro E(x) + \mathcal{I} : x \in E\}$ of $E$ in
$\TROu(E)/\mathcal{I}$ by $\tilde{E}_{\mathcal{I}}$, noting that the
latter is a completely isometric copy of $E_{\mathcal{I}}$.

\begin{theorem}
\label{thm:201206510}
Let $E$ be a universally reversible \JCstar-triple. Then 
$\mathcal{I} \mapsto E_{\mathcal{I}}$
is a bijective correspondence between the operator space ideals
of $\TROu(E)$
and the \JC -operator space structures of $E$.
\end{theorem}

\begin{proof}
We know that the stated correspondence is surjective. By
\cite[Theorem 6.5]{BunceFeelyTimoney}
given an operator space ideal $\mathcal{J}$
of $\TROu(E)$ there is a
largest operator space ideal 
$\mathcal{I}$ of $\TROu(E)$ with
$E_{\mathcal{I}}=  E_{\mathcal{J}}$. To  establish
injectivity it is enough to show that this largest 
$\mathcal{I}$ coincides with
$\mathcal{J}$.
Suppose on the contrary that $\mathcal{I}$, $\mathcal{J}$
are operator space ideal of $\TROu(E)$
with $\mathcal{J}$ strictly contained in $\mathcal{I}$ but with
$E_{\mathcal{I}}  =  E_{\mathcal{J}}$.
Simplifying notation we shall regard $E$ as a subtriple of 
$\TROu(E)$. Put $K
=  E \cap (\mathcal{I} + φ\Phi(\mathcal{I}))$.
Then, by
Lemma~\ref{lem:idsTROuE}
and
Proposition~\ref{prop:idsTROuE},
we have $K  =  \{x +
\Phi(x): x ɛ \in \mathcal{I} \}$ and
$\TROu(K)
= \mathcal{I} + φ\Phi(\mathcal{I})
= \mathcal{I} \oplus_\infty φ\Phi(\mathcal{I})$.
In particular,
$\mathcal{I}$ and $\mathcal{J}$ are operator space ideals of 
$\TROu(K)$
and we may regard 
$\tilde{K}_{\mathcal{I}}$
and 
$\tilde{K}_{\mathcal{J}}$
as operator subspaces of $\TROu(K)/\mathcal{I}$
and $\TROu(K)/\mathcal{J}$, respectively. We shall
show that the linear isometry $\pi : \tilde{K}_{\mathcal{I}} \to
\tilde{K}_{\mathcal{J}}$
given by $\pi(x + \mathcal{I})  =
x + \mathcal{J}$, for all $x$ in $K$, is not a complete contraction,
implying $K_{\mathcal{I}} \neq
K_{\mathcal{J}}$
and hence 
$E_{\mathcal{I}} \neq   E_{\mathcal{J}}$.
We have $\tilde{K}_{\mathcal{I}}  =  (K +\mathcal{I})/\mathcal{I}
=  \{φ\Phi(x) + \mathcal{I}: x \inɛ \mathcal{I} \}  =
\TROu(K)/\mathcal{I}   ≅
\cong \Phiφ(\mathcal{I})$
(as TROs).
The map β$\beta \colon \tilde{K}_{\mathcal{I}} \to
\Phi(\mathcal{I})$
 ($\Phi(x) + \mathcal{I} \mapsto ↦ \Phi(x)$)
is a TRO isomorphism and hence a complete isometry.
Define ψ$\psi \colon \TROu(K)/\mathcal{J} \to → \mathcal{I}/\mathcal{J}$
by 
$\psi(a + \Phi(b) + \mathcal{J})  =  a + \mathcal{J}$,
whenever $a, b ɛ\in \mathcal{I}$. Then $\psi$
ψ is a TRO homomorphism restricting  to a surjective complete
contraction ϒ
$\Upsilon \colon \tilde{K}_{\mathcal{J}} \to
 \mathcal{I}/\mathcal{J}$. We note that 
 $\mathcal{I}/\mathcal{J}$
 cannot be abelian, by
Lemma~\ref{lem:20120657} (\ref{lem:20120657c})
together with Proposition~\ref{prop:20120659}
(\ref{prop:20120659a}).
On the other hand the composition $\Upsilon \circ \pi \circ
\beta^{-1}$
\[
\Phi(\mathcal{I}) \to 
\tilde{K}_{\mathcal{I}} 
\to
\tilde{K}_{\mathcal{J}} 
\to
\mathcal{I}/\mathcal{J}
\]
which sends φ$\Phi(x)$ to $x + \mathcal{J}$
for all $x$ in $\mathcal{I}$, is a TRO
antihomomorphism and so, by
Lemma~\ref{lem:TTop} (\ref{lem:TTopb})
cannot be completely
contractive. Hence, $\pi$ is not a complete contraction.
\end{proof}

For the definition and
properties of the triple envelope
$(\mathcal{T}(X), j)$
of an  an operator space $X$
we refer to
\cite{Hamana99}
and
\cite[Chapter 8]{BlecherleMerdy}, and again to the latter for a justification that it be
regarded as the noncommutative Shilov boundary of $X$.
 We should point out a clash of usage: the terms ``triple system''
and ``triple homomorphism''  used in 
\cite{BlecherleMerdy,Hamana99}
should be read as
``TRO''  and  ``TRO homomorphism''. In particular, $\mathcal{T}(X)$
is a TRO.
The following is immediate from
Theorem~\ref{thm:201206510}
and
\cite[Proposition 1.2]{BunceTimoneyII}.

\begin{corollary}
\label{coroll:ShilovBdry}
\begin{enumerate}[(a)]
\item
If $E$ is a universally reversible \JCstar-triple and 
$\mathcal{I}$ is an operator space ideal of 
$\TROu(E)$ then the triple envelope of the operator space
$E_{\mathcal{I}}$
is $(\TROu(E)/\mathcal{I}, j)$,  where 
$j \colon E_{\mathcal{I}} \to → \TROu(E)/\mathcal{I}$
($x↦\mapstoα\utro E (x) + \mathcal{I}$).

\item 
If $E$ is a universally reversible \JCstar-subtriple of a 
\Cstar-algebra $A$ (regarded as an operator subspace of $A$) then
\begin{enumerate}[(i)]
\item
the triple envelope of $E$ is $(\TRO(E), \mathrm{inclusion})$;

\item
every complete isometry from $E$ onto a \JCstar-subtriple $F$ of 
$\BH$
extends to a TRO isomorphism from $\TRO(E)$ onto $\TRO(F)$.
\end{enumerate}
\end{enumerate}
\end{corollary}
                               
\begin{theorem}
\label{thm:20120622512}
The following are equivalent for  a universally
reversible \JCstar-triple $E$.
\begin{enumerate}[(a)]
\item 
\label{thm:20120622512a}
$E$ has a unique \JC-operator space structure.

\item 
\label{thm:20120622512b}
$\TROu(E)$ has no nonzero operator space ideal.

\item 
\label{thm:20120622512c}
$E$ has no ideal linearly isometric to a nonabelian TRO.

\item 
\label{thm:20120622512d}
If  $\pi : E \to \BH$
is an injective triple homomorphism, then
$(\TRO(\pi(E)), \pi)$ is the universal TRO of $E$.

\end{enumerate}
\end{theorem}

\begin{proof}
The equivalences
(\ref{thm:20120622512a}) $\iff$ (\ref{thm:20120622512b})
and
(\ref{thm:20120622512b}) $\iff$ (\ref{thm:20120622512d})
are immediate
from Theorem~\ref{thm:201206510}
and properties of $\TROu( \cdot )$, respectively.

(\ref{thm:20120622512a}) $\Rightarrow$ (\ref{thm:20120622512c}).
Suppose that $E$ has an ideal $I$ linearly isometric to a
nonabelian TRO. Via Lemma~\ref{lem:20120658}
$I$ has a nonzero ideal $I_1$
linearly isometric
to a universally reversible TRO without one dimensional representations.
By Theorem~\ref{thm:TROurcharact}
the latter TRO satisfies the conditions of Theorem~\ref{thm:TROuofT}
so that $I_1$, and hence $E$, has at least three \JC-operator
structures.

(\ref{thm:20120622512c}) $\Rightarrow$ (\ref{thm:20120622512b}).
Suppose $\mathcal{I}$
is a nonzero operator space ideal of
$\TROu(E)$. Then $\mathcal{I}$
is nonabelian by Proposition~\ref{prop:20120659}
and is linearly
isometric to an ideal of $E$ by Proposition~\ref{prop:idsTROuE},
proving the implication.
\end{proof}
□

\begin{remarks}
Since the nonzero operator space ideals occur in
pairs (Lemma \ref{lem:idsTROuE} (\ref{lem:idsTROuEa})),
it follows from Theorem~\ref{thm:201206510}
that the number of distinct
\JC-operator spaces of a universally reversible JC*-triple must be odd or
infinite.
In
\cite[Remark 6.9]{BunceFeelyTimoney}
we proved that for a projection $e$ of rank $\geq 2$ in $\BH$, the
universally reversible Cartan factor $\BH e$ has three or
$2|\alpha| + 5$
\JC-operator space structures according to whether
the rank of $e$ is finite or an infinite cardinal $\aleph_\alpha$
respectively, with $|\alpha|$ denoting the cardinality of
the ordinal segment $[0, \alpha)$
so that, as $|\alpha|$
varies over finite cardinalities all odd numbers $\geq 3$ arise
(the case $\alpha > 0$ was mis-stated in
\cite[Remark 6.9 (iii)]{BunceFeelyTimoney}).

On the other hand, if $E$ is a universally reversible \JCstar-triple with
no Cartan factor representations 
$\pi \colon E \to C$
such that $C$ has the form
$\BH e$, then $E$ must satisfy
condition 
(\ref{thm:20120622512c})
of
Theorem~\ref{thm:20120622512}
and so must have a unique \JC -operator space structure.

Another corollary of
Theorem~\ref{thm:20120622512}
is that
a simple
universally reversible \JCstar-triple $E$  either has a unique 
\JC-operator space structure
or it has exactly three.
By
Remark~\ref{remark:simpleUR}, this together with
\cite[Proposition 2.4 and Theorem 3.7]{BunceTimoneyII}
and \cite[\S6]{BunceTimoney3}, accounts for the enumeration of the \JC-operator space
structures of all simple \JCstar-triples.
\end{remarks}

Consider a triple homomorphism $\pi \colon E \to \mathcal{B}(K)$ where $E$
is a universally reversible \JCstar-subtriple of $\BH$.
If $E$ has no ideals isometric to a nonabelian TRO, then by
Theorem~\ref{thm:20120622512}
((\ref{thm:20120622512c}) $\Rightarrow$ 
(\ref{thm:20120622512d}))
$\pi$ extends to a TRO homomorphism on $\TRO(E) (= \TROu(E))$ and so
is completely contractive, and further is completely isometric (onto
its range) if $\pi$ is injective.
By contrast, for TROs with few nonzero Hilbert space representations we
have the following consequence of Theorem~\ref{thm:TROuofT} (which
should be compared with the \Cstar-algebra result
\cite[Corollary 4.6]{HancheOlsenJC}).

\begin{proposition}
\label{prop:triplehomdecomp}
Let $T$ be a TRO with no nonzero Hilbert space representations other
(possibly) than of dimension two.
Let $\pi \colon T \to \BH$ be a triple homomorphism.
Then there exist $\pi_1, \pi_2 \colon T \to \BH$ where $\pi_1$ is a
TRO homomorphism and $\pi_2$ is a TRO antihomomorphism such that $\pi
= \pi_1 + \pi_2$ and $\pi_1(T) \perp \pi_2(T)$.
\end{proposition}

\begin{proof}
Supposing $T \subseteq \mathcal{B}(K)$ and $x \mapsto x^t$ to be a
transposition of $\mathcal{B}(K)$
Theorem~\ref{thm:TROuofT}
implies that there is a TRO homomorphism $\tilde{\pi} \colon T \oplus
T^t \to \BH$ such that $\pi(x) = \tilde{\pi}(x \oplus x^t)$ for all $x
\in T$. Defining $\pi_1, \pi_2 \colon T \to \BH$ by $\pi_1(x) =
\tilde{\pi}(x \oplus 0)$ and $\pi_2(x) = \tilde{\pi}(0 \oplus x^t)$,
the result follows.
\end{proof}

We remark that by the results
of \cite[\S3 and \S4]{BunceTimoney3}, with the
notation of
Proposition~\ref{prop:triplehomdecomp}, there is a central projection
$z$ in the right von Neumann algebra of $\pi(T)$ such that $\pi_1(x) =
\pi(x) z$ and $\pi_2(z) = \pi(x)(1-z)$ for all $x \in T$.

\begin{corollary}
\label{coroll:TROhomORantiHom}
Let $T$ be as in Proposition~\ref{prop:triplehomdecomp} and let
$\pi \colon T \to \BH$ be a triple
homomorphism such that the weak*-closure $\overline{\TRO(\pi(T))}$ of
$\TRO(\pi(T))$ is a \Wstar-TRO factor. Then $\pi$ is either a TRO
homomorphism or a TRO antihomomorphism.
\end{corollary}

\begin{proof}
The factor condition implies that $\TRO(\pi(T))$ cannot contain a
nontrivial pair of orthogonal ideals. Thus, in the notation of
Proposition~\ref{prop:triplehomdecomp}, either $\pi_1$ is trivial or
$\pi_2$ is, whence the assertion.
\end{proof}

\begin{corollary}
\label{coroll:compcontractive}
Let $\pi \colon T \to \BH$ be a completely contractive triple
homomorphism where $T$ is as in
Proposition~\ref{prop:triplehomdecomp}.
Then $\pi$ is a TRO homomorphism.
\end{corollary}

\begin{proof}
By Proposition~\ref{prop:triplehomdecomp} and its proof, 
$\pi_1(T)$ and $\pi_2(T)$ are triple ideals of $\pi(T)$ so that
$\TRO(\pi(T)) = \TRO(\pi_1(T)) \oplus_\infty \TRO(\pi_2(T)) = \pi_1(T)
\oplus_\infty \pi_2(T)$,
and so $\pi_2$ must be completely contractive. Hence $\pi_2(T)$ is
abelian by
Proposition~\ref{lem:TTop} (\ref{lem:TTopb}), Hence $\pi_2(T) = \{0\}$
as otherwise $T$ would have a one dimensional (Hilbert space)
representation, proving the result.
\end{proof}

\begin{lemma}
\label{lem:WTROLRDECOMP}
If $T$ is a 
a \Wstar-TRO, 
there exist centrally orthogonal projections $e$ and $f$ in a von
Neumann algebra $W$ such that
$T$ is TRO isomorphic to the direct sum $eW + W f$.
If $T$ is a factor,
then $T$ is TRO isomorphic to a weak*-closed one-sided ideal in
a \Wstar-algebra factor.
\end{lemma}

\begin{proof}
Up to TRO  isomorphism, a \Wstar-TRO $T$
has the form $T = gMh$ where $g$ and
$h$ are projections in a von Neumann algebra $M$
\cite[\S2]{EffrosOzawaRuan}.
By comparison theory, there exist projections $e$, 
$f$ and $z$ in $M$, with $z$ central, such that 
\[
gz \sim e \leq hz \mbox{ and } h(1-z) \sim f \leq g(1-z)
\]
For $W_1 = zhMh$ and $W_2 = (1-z)gMg$, centrality of $z$ implies that
$Tz= (gz)Mh$ is TRO isomorphic to $pMh = e(hMh)z = eW_1$ (see proof of
Lemma~\ref{SquareRev}) and similarly
that $T(1-z)$ is TRO isomorphic to $W_2f$.
The result follows from putting $W = W_1 \oplus W_2$.

If $T$ is a factor, one summand must be zero and if (say) $T = e W$
we may assume that $e$ has central cover 1, in which case $W$ must
be a factor.
\end{proof}

The examples $x \mapsto x^t$ on $\BH$ and $x \oplus y^t \mapsto y \oplus x^t$ on $\BH e
\oplus e^t \BH$, where $e$ is a projection of rank strictly less than the dimension of
$H$, show that neither of the two conditions imposed upon $T$ in the next result can be
removed.

\begin{theorem}
\label{thm:WstarFactorCI}
Let $\pi \colon T \to T$ be a surjective linear isometry where $T$ is
a \Wstar-TRO factor not linearly isometric to a \Cstar-algebra. Then
$\pi$ is a complete isometry.
\end{theorem}

\begin{proof}
By Lemma~\ref{lem:WTROLRDECOMP} we may
suppose that $T$ is $eW$ where $e$ is a nonzero projection in
a von Neumann algebra factor $W$. (The argument for the left ideal case
is similar.)
We may further suppose that $T \subset \BH$ and that $x \mapsto x^t$
is a transposition of $\BH$ with $e^t = e$.

If there is a nonzero Hilbert space representation ψ$\psi \colon  eW
\to K$, then the induced map from the factor $eWe  =
(eW)_2(e)$ to $K_2(ψ\psi (e))$
is a one-dimensional Jordan *-homomorphism,
implying that $e$ is a minimal projection and $W$ is a type I factor: in
which case, $eW$ is (linearly isometric to) a Hilbert space and the
result follows from \cite[Proposition 1.5]{BunceTimoneyII}.

Thus we may suppose that $eW$ has no nonzero Hilbert space
representations and so satisfy the conditions of
Proposition~\ref{prop:triplehomdecomp}.

Therefore, by
Corollary~\ref {coroll:TROhomORantiHom},
$\pi$ is a TRO isomorphism or a TRO antiautomorphism.
Assume the latter. Then there is a TRO isomorphism
$\psi \colon W^t e \to eW$.
Putting $u = \psi(e)$ we have
\[
\psi(a) = \psi(a e^* e) = \psi(a) u^* u
\]
for all $a \in W^t e$, giving
$eW = eW u^* u$, hence
$WeW = WeW u^* u$ and therefore
$u^* u = 1$ because $W$ is a factor.
Since $u u^* \leq e$, this implies that $e \sim 1$ and therefore
that $eW$ is TRO isomorphic to $W$, a contradiction.
Therefore, $\pi$ is a TRO isomorphism.
\end{proof}

\def\cprime{$'$} \def\polhk#1{\setbox0=\hbox{#1}{\ooalign{\hidewidth
  \lower1.5ex\hbox{`}\hidewidth\crcr\unhbox0}}}
  \def\polhk#1{\setbox0=\hbox{#1}{\ooalign{\hidewidth
  \lower1.5ex\hbox{`}\hidewidth\crcr\unhbox0}}}

\end{document}